\documentclass[final]{article}

\usepackage{tikz}
\usetikzlibrary{shapes}
\usepackage{enumitem,centernot,comment}
\usepackage{graphicx,xcolor}
\usepackage{url}
\usepackage{amssymb,amsmath}
\usepackage{colortbl}
\usepackage{bm}

\usepackage{savesym}
\savesymbol{hyperpage}
\usepackage{hyperref}
\restoresymbol{HR}{hyperpage}
\hypersetup{
    colorlinks,
    linkcolor={red!50!black},
    citecolor={blue!50!black},
    urlcolor={blue!80!black},
  linkbordercolor={red!50!black},% hyperlink border will be red
}

\usepackage[color,notref,notcite]{showkeys}
\definecolor{labelkey}{rgb}{0.6,0,1}

\usepackage[normalem]{ulem}
\newcounter{corr}
\definecolor{violet}{rgb}{0.580,0.,0.827}
\newcommand{\corr}[3]{\typeout{Warning : a correction remains in page
\thepage}
				\stepcounter{corr}        
				{\color{blue}\ifmmode\text{\,\sout{\ensuremath{#1}}\,}\else\sout{#1}\fi}
       {\color{red}#2}
       {\color{violet} #3}}

\usepackage{constants}
\newconstantfamily{ctrcst}{symbol=C,reset={paper}}

\newcommand{\mathbi}[1]{{\boldsymbol #1}}
\newcommand{\norm}[2]{\| #1 \|_{#2}}

\usepackage{subcaption}
\usepackage{pgfplots,pgfplotstable}
\usetikzlibrary{calc,external}
\usetikzlibrary{spy}
\usetikzlibrary{cd}
\tikzcdset{every diagram/.append style = {font=\normalsize}}
\tikzcdset{every label/.append style = {font = \small}}
\tikzset{ shorten <>/.style={ shorten >=#1, shorten <=#1 } }

\newcommand{\logLogSlopeTriangle}[5]
{
    \pgfplotsextra
    {
        \pgfkeysgetvalue{/pgfplots/xmin}{\xmin}
        \pgfkeysgetvalue{/pgfplots/xmax}{\xmax}
        \pgfkeysgetvalue{/pgfplots/ymin}{\ymin}
        \pgfkeysgetvalue{/pgfplots/ymax}{\ymax}

        \pgfmathsetmacro{\xArel}{#1}
        \pgfmathsetmacro{\yArel}{#3}
        \pgfmathsetmacro{\xBrel}{#1-#2}
        \pgfmathsetmacro{\yBrel}{\yArel}
        \pgfmathsetmacro{\xCrel}{\xArel}
        \pgfmathsetmacro{\lnxB}{\xmin*(1-(#1-#2))+\xmax*(#1-#2)} % in [xmin,xmax].
        \pgfmathsetmacro{\lnxA}{\xmin*(1-#1)+\xmax*#1} % in [xmin,xmax].
        \pgfmathsetmacro{\lnyA}{\ymin*(1-#3)+\ymax*#3} % in [ymin,ymax].
        \pgfmathsetmacro{\lnyC}{\lnyA+#4*(\lnxA-\lnxB)}
        \pgfmathsetmacro{\yCrel}{\lnyC-\ymin)/(\ymax-\ymin)}

        \coordinate (A) at (rel axis cs:\xArel,\yArel);
        \coordinate (B) at (rel axis cs:\xBrel,\yBrel);
        \coordinate (C) at (rel axis cs:\xCrel,\yCrel);

        \draw[#5]   (A)-- node[pos=0.5,anchor=north] {\scriptsize{1}}
                    (B)-- 
                    (C)-- node[pos=0.,anchor=west] {\scriptsize{#4}} %% node[pos=0.5,anchor=west] {#4}
                    cycle;
    }
}

\newtheorem{theorem}{Theorem}[section]
\newtheorem{definition}{Definition}[section]
\newtheorem{proposition}[theorem]{Proposition}
\newtheorem{lemma}[theorem]{Lemma}

\newtheorem{remark}[theorem]{Remark}

\newenvironment{proof}{\noindent {\sc Proof.} }{$\square$ }
\begin{document}

\title{Non-conforming finite elements on polytopal meshes}
\author{J\'er\^ome Droniou\thanks{School of Mathematics, Monash University, Melbourne, Australia, \tt{jerome.droniou@monash.edu}}, Robert Eymard\thanks{LAMA, Universit\'e Gustave Eiffel, UPEM, F-77447 Marne-la-Vall\'ee, France, \tt{robert.eymard@u-pem.fr}}, Thierry Gallou\"et and Rapha\`ele Herbin\thanks{Institut de Math\'ematiques de Marseille, Aix-Marseille Universit\'e, Marseille, France, \tt{thierry.gallouet@univ-amu.fr, raphaele.herbin@univ-amu.fr}}}

\maketitle

\abstract{
In this work we present a generic framework for non-conforming finite elements on polytopal meshes, characterised by elements that can be generic polygons/polyhedra. 
We first present the functional framework on the example of a linear elliptic problem representing a single-phase flow in porous medium. 
This framework gathers a wide variety of possible non-conforming methods, and an error estimate is provided for this simple model. 
We then turn to the application of the functional framework to the case of a steady degenerate elliptic equation, for which a mass-lumping technique is required; here, this technique simply consists in using a different --piecewise constant-- function reconstruction from the chosen degrees of freedom. 
A convergence result is stated for this degenerate model. 
Then, we introduce a novel specific non-conforming method, dubbed Locally Enriched Polytopal Non-Conforming  (LEPNC). 
These basis functions comprise functions dedicated to each face of the mesh (and associated with average values on these faces), together with functions spanning the local $\mathbb{P}^1$ space in each polytopal element. 
The analysis of the interpolation properties of these basis functions is provided, and mass-lumping techniques are presented. 
Numerical tests are presented to assess the efficiency and the accuracy of this method on various examples. 
Finally, we show that generic polytopal non-conforming methods, including the LEPNC, can be plugged into the gradient discretization method framework, which makes them amenable to all the error estimates and convergence results that were established in this framework for a variety of models.}

\section{Introduction}

Problems involving elliptic partial differential equations are often efficiently approximated by the Lagrange finite element method, yielding an approximation of the unknown functions at the nodes of the mesh. 
In some cases, it may however be more interesting to approximate the unknown functions at the centre of the faces of the mesh. 
This is for example the case for the Stokes and Navier-Stokes problems, where an approximation of the velocity of a fluid at the faces of the mesh leads to an easy way to take into account the conservation of fluid mass in each element. 
This property is the basis of the success of the Crouzeix-Raviart approximation for the incompressible Stokes and Navier-Stokes equations; see the seminal paper by Crouzeix and Raviart \cite{crouzeix-raviart-73}, and recent extensions including linear elasticity \cite{DPL15}.

Another situation for which approximating functions at the face centre is highly relevant is found in underground flows in heterogeneous porous media. 
Several coupled models require to simultaneously solve an elliptic equation associated with the pressure of the fluid, and equations associated with the transport of species by different mechanisms including convection with the displacement of the fluid, diffusion/dispersion mechanisms, and chemical and thermodynamic reactions. 
In such cases, the accuracy of the model on relatively coarse meshes can only be obtained if the elements of the mesh are homogeneous, in order to compute the flows in the high permeability zones as precisely as possible, without integrating in these zones some porous volume belonging to low permeability zones. 
Non-conforming methods with unknowns at the face naturally lead to finite volume properties on the elements, which are useful for the discretisation of such coupled equations. 
Note that non-conforming methods are in some way strongly linked with mixed finite elements on the same mesh, in the sense that the matrix resulting from the mixed hybrid condensed formulation for the Raviart-Thomas finite element is the same as the non conforming P1 finite element \cite{chen1996equi,vohralik2007mixed}.

The aim of this paper is twofold.

On one hand, we wish to provide a general framework for the functional basis of non-conforming methods on polytopal meshes. 
Polytopal meshes have elements that can be generic polygons or polyhedra; they have gained considerable interest because they allow to mesh complex geometries or match specific underground features. 
For example, in the framework of petroleum engineering, general hexahedra have been used for several years; numerical developments for the computation of porous flows on such grids may be found in \cite{aav2002intro}, for multi-point flux approximation finite volume methods for instance, in \cite{wheeler2006multi} for multi-point mixed approximations, or in \cite{mfdrev} for mimetic finite difference methods. 
The use of polytopal meshes for underground flows has motivated so many papers that it is impossible to give an exhaustive list; we refer the reader to the introduction of \cite{hho-book} for a thorough literature review on the topic.

Let us focus on the non-conforming finite element method for second order differential forms, described on simplicial meshes for example in \cite{ciarlet,zien2014fin}. 
By non-conforming finite element method we refer to a method such that:
\begin{itemize}
\item the restriction to each element of the approximate solution belongs to $H^1$, 
\item the approximate solution can be discontinuous at the common face between two elements everywhere, but some weak (averaged or at a certain point on the face) continuity is imposed,
\item the approximate gradient is defined as the broken gradient, which is locally (i.e. on each cell) the gradient of the function.  
\end{itemize}

The mathematical properties behind the nature of the continuity conditions at the faces, needed for the convergence of the method, are sometimes called the ``patch test''  \cite{stum1979genpatch}.
In Section \ref{sec:nc.approx.lin}, we revisit these properties, plugging all the non-conforming methods into a broken continuous $H^1$ space defined on a general polytopal mesh. 
We thus obtain in Section \ref{sec:nonconf.general}, a general error estimate in the case of a linear elliptic equation in heterogeneous and anisotropic cases. Section \ref{sec:nc.approx.lin} can be read as a simple introduction, using a basic linear model as illustration, to generic non-conforming finite-element methods on polytopal meshes.

In Section \ref{sec:nc.ml}, we explore the use of these methods on a more challenging model, which is however very relevant to applications in geosciences: a nonlinear degenerate elliptic equation of the Stefan or porous medium equation type.
We introduce in Section \ref{sec:ml.generic} a mass lumping technique, which is mandatory for designing robust numerical schemes for this model. 

We then focus, in Section \ref{sec:bpncfe}, on a new specific non-conforming approximation on general polytopal meshes, called the Locally Enriched Polytopal Non-Conforming (LEPNC) method. 
This method is based on the $H^1$ piecewise approximation, imposing the continuity of the mean value on the interfaces. 
The advantage of the method presented here is its robustness, which is not the case for other possible simpler methods, such as choosing on each cell polynomials of degree $k$ with dim $\mathbb{P}^k(\mathbb R^d)$ larger than or equal to the number of faces of the polytopal cell (this condition is necessary to obtain a decent approximation, see e.g.\ the hexagonal example of Section \ref{sec:bpncfe}, but it is not sufficient to solve robustness issues, see Remark \ref{rem:need.bubble}).
In particular, the LEPNC method allows for hanging nodes which frequently occur when meshing two different zones such as in domain decomposition methods. 
Another important feature of the finite element method presented here is that it can be used together with $\mathbb{P}^1$ nonconforming finite elements on simplicial parts of the mesh.
The LEPNC basis functions are described in Sections \ref{sec:const.ncPoly}--\ref{NCpoly:sec:spaceV}, and the approximation properties of the method are detailed in Section \ref{sec:appprobpnc}. The convergence theorems for the LEPNC method are given in Section \ref{sec:cvgcelepnc}.
 Various numerical tests are then proposed in Section \ref{sec:numer}, showing the accuracy and the efficiency of this method on problems presenting some complex features.

Section \ref{sec:prop.NC} covers the generic analysis of the convergence of non-conforming methods, which is encompassed in the framework of the Gradient Discretization method \cite{gdm}. 
Some perspectives are then drawn in Section \ref{sec:conclusion}.

\section{Principles of polytopal non-conforming approximations}\label{sec:nc.approx.lin}

\subsection{The model: linear single-phase incompressible flows in po\-rous media}

The principles of a generic polytopal non-conforming method are first presented on the following linear model of pressure for a single-phase incompressible flow in a porous medium:

\begin{equation}\label{eq:linear}
\left\{\begin{array}{ll}
-\mathop{\rm div}(\Lambda\nabla \bar u) = f+\mathop{\rm div}(\bm{F})&\quad\mbox{ in }\Omega\,,\\
\bar u=0&\quad\mbox{ on }\partial\Omega,
\end{array}\right.
\end{equation}
with the following assumptions on the data:
\begin{subequations}
\begin{align}
  \bullet~ & \Omega \mbox{ is a polytopal open subset of $\mathbb{R}^d$ ($d\in\mathbb{N}^\star$)}, \label{hypomega} \\
\bullet~ & \Lambda \hbox{ is a measurable function from } \Omega  \hbox{ to the set of $d\times d$  }\nonumber\\
&\mbox{symmetric matrices and there exists $\underline{\lambda},\overline{\lambda}>0$ such that,}\nonumber\\
&\mbox{for a.e.\ $\mathbi{x} \in\Omega$, $\Lambda(\mathbi{x})$ has eigenvalues in $[\underline{\lambda},\overline{\lambda}]$,}
\label{hyplambda}\\
\bullet~  & f \in L^2(\Omega)\,,\;\bm{F}\in L^2(\Omega)^d. \label{hypfg}
\end{align}
\label{hypglin}
\end{subequations}
We note in passing that a polytopal open set is simply a bounded polygon (if $d=2$) or polyhedron (if $d=3$) without slit, that is, it lies everywhere on one side of its boundary; see \cite[Section 7.1.1]{gdm} for a more formal definition.

The solution to \eqref{eq:linear} is to be understood in the standard weak sense:
\begin{equation}\begin{array}{l}
\mbox{Find $\bar u \in H^1_0(\Omega)$ such that, }\forall v \in H^1_0(\Omega),\\
\displaystyle\int_\Omega \Lambda\nabla\bar u\cdot\nabla v d\mathbi{x}  
= \int_\Omega fv d\mathbi{x} 
-\int_\Omega \bm{F}\cdot\nabla v d\mathbi{x}.
\end{array}\label{ellgenf}\end{equation}

\subsection{Polytopal non-conforming method}\label{sec:nonconf.general}

A polytopal non-conforming scheme for \eqref{ellgenf} is obtained by replacing the continuous space $H^1_0(\Omega)$ in this weak formulation by a finite-dimensional subspace of a ``non-conforming Sobolev space''. 
Let us first give the definition of polytopal mesh we will be working with; this definition is a simplified version of \cite[Definition 7.2]{gdm}.

\begin{definition}[Polytopal mesh]\label{def:polymesh}~
Let $\Omega$ satisfy Assumption \eqref{hypomega}. 
A polytopal mesh of $\Omega$ is a triplet $\mathfrak{T} = (\mathcal{M},\mathcal F,\mathcal{P})$, where:
\begin{enumerate}
\item \label{item:mesh} $\mathcal{M}$ is a finite family of non empty connected polytopal open disjoint subsets of $\Omega$ (the ``cells'') such that $\overline{\Omega}= \displaystyle{\cup_{K \in \mathcal{M}} \overline{K}}$.
For any $K\in\mathcal{M}$, $\partial K  = \overline{K}\setminus K$ is the boundary of $K$, $|K|>0$ is the measure of $K$ and $h_K$ denotes the diameter of $K$, that is the maximum distance between two points of $\overline{K}$.

\item $\mathcal F=\mathcal F_{\rm int}\cup\mathcal F_{\rm ext}$ is a finite family of disjoint subsets of $\overline{\Omega}$ (the ``faces'' of the mesh -- ``edges'' in 2D), such that any $\sigma\in\mathcal F_{\rm int}$ is contained in $\Omega$ and any $\sigma\in\mathcal F_{\rm ext}$ is contained in $\partial\Omega$.
Each $\sigma\in\mathcal F$ is assumed to be a non empty open subset of a hyperplane of $\mathbb{R}^d$, with a strictly positive $(d-1)$-dimensional measure $|\sigma|$, and a relative interior $\overline{\sigma}\backslash\sigma$ of zero $(d-1)$-dimensional measure. 
We denote by $\overline{\mathbi{x}}_\sigma$ the centre of mass of $\sigma$.
Furthermore, for  all $K \in \mathcal{M}$, there exists  a subset $\mathcal F_K$ of $\mathcal F$ such that $\partial K  = \displaystyle{\cup_{\sigma \in \mathcal F_K}} \overline{\sigma} $. 
We set $\mathcal{M}_\sigma = \{K\in\mathcal{M}\,:\, \sigma\in\mathcal F_K\}$ and assume that, for all $\sigma\in\mathcal F$, either  $\mathcal{M}_\sigma$ has exactly one element and then $\sigma\in \mathcal F_{\rm ext}$, or $\mathcal{M}_\sigma$ has exactly two elements and then $\sigma\in \mathcal F_{\rm int}$.
For $K \in \mathcal{M}$ and $\sigma \in \mathcal F_K$, $\mathbi{n}_{K,\sigma}$ is the (constant) unit vector normal to $\sigma$ outward to $K$.

\item $\mathcal{P}=(\mathbi{x}_K)_{K \in \mathcal{M}}$ is a family of points of $\Omega$ such that $\mathbi{x}_K\in K$ for all  $K\in\mathcal{M}$. 
We denote by $d_{K,\sigma}$ the signed orthogonal distance between $\mathbi{x}_K$ and $\sigma \in \mathcal F_K$ (see Fig. \ref{fig.dksigma}), that is: 
\begin{equation}
 d_{K,\sigma} = (\mathbi{x}  - \mathbi{x}_K) \cdot \mathbi{n}_{K,\sigma},\mbox{ for all } \mathbi{x}  \in \sigma. \label{def.dcvedge}
\end{equation}
(Note that $(\mathbi{x}  - \mathbi{x}_K) \cdot \mathbi{n}_{K,\sigma}$ is constant for $\mathbi{x}  \in \sigma$.)
We then assume that each cell $K\in\mathcal{M}$ is strictly star-shaped with respect to $\mathbi{x}_K$, that is $d_{K,\sigma}> 0$ for all $\sigma\in\mathcal F_K$. 
This implies that for all $\mathbi{x} \in K$, the line segment $[\mathbi{x}_K,\mathbi{x} ]$ is included in $K$.

For all $K\in\mathcal{M}$ and $\sigma\in\mathcal F_K$, we denote by $D_{K,\sigma}$ the pyramid with vertex $\mathbi{x}_K$ and basis $\sigma$, that is
\begin{equation}
D_{K,\sigma}=\{ t \mathbi{x}_K +(1-t) \mathbi{y}\,:\, t\in (0,1), \ \mathbi{y}\in \sigma\}.
\label{bldefdks}\end{equation}
We denote, for all $\sigma\in\mathcal F$, $D_\sigma = \bigcup_{K\in\mathcal{M}_\sigma} D_{K,\sigma}$ (this set is called the ``diamond'' associated with the face $\sigma$, and for obvious reasons $D_{K,\sigma}$ is also referred to as an ``half-diamond'').

\end{enumerate}
The size of the polytopal mesh is defined by $h_{\mathcal M} = \sup\{h_K\,:\, K\in \mathcal{M}\}$ and 
the mesh regularity parameter $\gamma_\mathfrak{T}$ is defined by:
\begin{equation}
\gamma_\mathfrak{T} = \max_{K\in\mathcal{M}}
\left(\max_{\sigma\in\mathcal F_K}\frac{h_K}{d_{K,\sigma}}+{\rm Card}(\mathcal F_K)\right)
+
 \max_{\sigma\in\mathcal F_{\rm int}\,,\;\mathcal{M}_\sigma=\{K,L\}} \left(\frac {d_{K,\sigma}} {d_{L,\sigma}}+ \frac {d_{L,\sigma}} {d_{K,\sigma}}\right). \label{def:reg.gamma}
\end{equation}
\end{definition}

\begin{figure}[htb]
\begin{center}
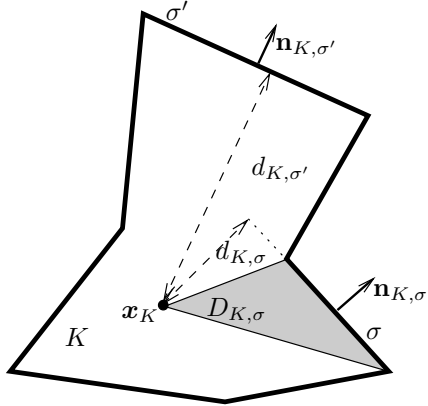
\caption{A cell $K$ of a polytopal mesh}
\label{fig.dksigma}
\end{center}
\end{figure}
We can now define the notion of non-conforming Sobolev space, which is built from the standard broken Sobolev space on a mesh by imposing some weak continuity property between the cells.

\begin{definition}[Non-conforming $H^1_0(\Omega)$ space]\label{chapncfe:def:ncspace}
Let  $\mathfrak{T} = (\mathcal{M},\mathcal F,\mathcal{P})$ be a polytopal mesh of $\Omega$ in the sense of Definition \ref{def:polymesh}. 
The non-conforming $H^1_0(\Omega)$ space on $\mathfrak{T}$, denoted by $H^1_{\mathfrak{T},0}$, is the space of all functions $w\in L^2(\Omega)$ such that:
\begin{enumerate}
 \item{} \emph{[$H^1$-regularity in each cell]} For all $K\in\mathcal{M}$, the restriction $w_{|K}$ of $w$ to $K$ belongs to $H^1(K)$. The trace of $w_{|K}$ on $\sigma\in\mathcal F_K$ is denoted by $w_{|K,\sigma}$.
 \item{} \emph{[Continuity of averages on internal faces]} For all $\sigma\in\mathcal F_{\rm int}$ with $\mathcal{M}_\sigma = \{K,L\}$,
 \begin{equation}
 \int_\sigma w_{|K,\sigma}  = \int_\sigma w_{|L,\sigma}.
 \label{ncfe:eq:sameav}\end{equation}
 \item{} \emph{[Homogeneous Dirichlet BC for averages on external faces]} For all $\sigma\in\mathcal F_{\rm ext}$ with $\mathcal{M}_\sigma = \{K\}$,
 \begin{equation}
 \int_\sigma w_{|K,\sigma} = 0.
 \label{ncfe:eq:avnull}\end{equation}
\end{enumerate}
If $w\in H^1_{\mathfrak{T},0}$, its ``broken gradient'' $\nabla_{\mathcal M} w$ is defined by
\[
\forall K\in\mathcal{M}\,,\;\nabla_{\mathcal M} w = \nabla (w_{|K})\mbox{ in $K$}
\]
and we set $\norm{w}{H^1_{\mathfrak{T},0}} := \norm{\nabla_{\mathcal M} w}{L^2(\Omega)^d}$.
\end{definition}

It can easily be checked that $\norm{{\cdot}}{H^1_{\mathfrak{T},0}}$ is indeed a norm on $H^1_{\mathfrak{T},0}$. 
The continuity \eqref{ncfe:eq:sameav} is a ``0-degree patch test'', and some functions in $H^1_{\mathfrak{T},0}(\Omega)$ are therefore not conforming (they do not belong to $H^1_0(\Omega)$). 
Actually, disregarding the boundary condition \eqref{ncfe:eq:avnull}, the non-conforming Sobolev space strictly lies between the classical Sobolev space $H^1(\Omega)$ and the fully broken Sobolev space $H^1(\mathcal M)=\{v\in L^2(\Omega)\,:\,v_{|K}\in H^1(K)\mbox{ for all $K\in\mathcal M$}\}$.

A polytopal non-conforming approximation of \eqref{ellgenf} is obtained by selecting a finite-dimensional subspace $V_{\mathfrak{T},0}\subset H^1_{\mathfrak{T},0}$, by replacing, in this weak formulation, the infinite-dimensional space $H^1_0(\Omega)$ by $V_{\mathfrak{T},0}$, and by using broken gradients instead of standard gradients:  
\begin{equation}\begin{array}{l}
\mbox{Find $u \in V_{\mathfrak{T},0}$ such that, }\forall v \in V_{\mathfrak{T},0},\\
\displaystyle\int_\Omega \Lambda\nabla_{\mathcal M} u\cdot\nabla_{\mathcal M} v d\mathbi{x}  
= \int_\Omega fv d\mathbi{x} 
-\int_\Omega \bm{F}\cdot\nabla_{\mathcal M} v d\mathbi{x}.
\end{array}\label{eq:nc.lin}\end{equation}
Since $\norm{{\cdot}}{H^1_{\mathfrak{T},0}}$ is a norm on $V_{\mathfrak{T},0}$, the Lax-Milgram theorem immediately gives the existence and uniqueness of the solution to \eqref{eq:nc.lin}.
The following error estimate is a straightforward consequence of the analysis carried out in Section \ref{sec:prop.NC} (see in particular Theorem \ref{th:gdm.error} and Proposition \ref{prop:est.NC.GDM}).

\begin{theorem}[Error estimates for polytopal non-conforming methods]\label{th:error.est}
We assume that the solution $\bar u$ of \eqref{ellgenf} and the data $\Lambda$ and $\bm{F}$ in Hypotheses \eqref{hypglin} are such that $\Lambda\nabla\bar u+\bm{F} \in H^1(\Omega)^d$.
Let $V_{\mathfrak{T},0}$ be a finite-dimensional subspace of $H^1_{\mathfrak{T},0}$ and let $u$ be the solution of the non-conforming scheme \eqref{eq:nc.lin}. 
Then, there exists $C>0$ depending only on $\Omega$, $\underline{\lambda},\overline{\lambda}$ in \eqref{hyplambda} and increasingly depending on $\gamma_\mathfrak{T}$ such that
\begin{equation}\label{eq:errorest.lin}
\norm{\bar u-u}{L^2(\Omega)}+\norm{\nabla \bar u-\nabla_{\mathcal M}u}{L^2(\Omega)^d}
\le C h_{\mathcal M} \norm{\Lambda\nabla\bar u+\bm{F}}{H^1(\Omega)^d} + C \min_{v\in V_{\mathfrak{T},0}}
\norm{\bar u-v}{H^1_{\mathfrak T,0}}.
\end{equation}
\end{theorem}

\begin{remark}[Role of the terms in \eqref{eq:errorest.lin}]
The term $Ch_{\mathcal M} \norm{\Lambda\nabla\bar u+\bm{F}}{H^1(\Omega)^d}$ in the right-hand side of \eqref{eq:errorest.lin} comes from the non-conformity of the space $V_{\mathfrak{T},0}$, and from the fact that an exact Stokes formula is not satisfied in this space (as measured by $W_{\mathcal D}$ in Section \ref{sec:gdm}). 
The minimum appearing in \eqref{eq:errorest.lin} measures the approximation properties of the space $V_{\mathfrak{T},0}$, as in the second Strang lemma \cite{strang-fix} (see $S_{\mathcal D}$ in Section \ref{sec:gdm}).
\end{remark}

\section{Application to a non-linear model: mass-lumping}\label{sec:nc.ml}

\subsection{Model: stationary Stefan/porous medium equation}\label{sec:stefanmodel}

We now consider the polytopal non-conforming approximation of a more challenging model, which encompasses the stationary versions of both the Stefan model and the porous medium equation:
\begin{equation}\label{eq:stefan.strong}
\left\{\begin{array}{ll}
\bar u-\mathop{\rm div}(\Lambda\nabla \zeta(\bar u)) = f+\mathop{\rm div}(\bm{F})&\quad\mbox{ in }\Omega\,,\\
\zeta(\bar u)=0&\quad\mbox{ on }\partial\Omega.
\end{array}\right.
\end{equation}
We still assume that \eqref{hypglin} holds and, additionally, that
\begin{equation}
\begin{aligned}
& \zeta:\mathbb{R}\to\mathbb{R}\mbox{ is non-decreasing, $\zeta(0)=0$ and}\\
&\exists C_1,C_2>0\mbox{ such that }|\zeta(s)|\ge C_1|s|-C_2\mbox{ for all $s\in\mathbb{R}$}.
\end{aligned}
\label{hyp:zeta}
\end{equation}
The weak form of \eqref{eq:stefan.strong} is
\begin{equation}\begin{array}{l}
\mbox{Find $\bar u \in L^2(\Omega)$ such that $\zeta(\bar u)\in H^1_0(\Omega)$ and, }\forall v \in H^1_0(\Omega),\\
\displaystyle\int_\Omega\left(\bar u v + \Lambda\nabla\zeta(\bar u)\cdot\nabla v\right) d\mathbi{x}  
= \int_\Omega fv d\mathbi{x} 
-\int_\Omega \bm{F}\cdot\nabla v d\mathbi{x}.
\end{array}\label{eq:stefan.weak}\end{equation}

\subsection{Mass-lumping}\label{sec:ml.generic}

As explained in the introduction of \cite{DE19} (see also Appendix B therein), using a standard (conforming or non-conforming) Galerkin approximation for \eqref{eq:stefan.weak} leads to a numerical scheme whose properties are difficult to establish. 
In particular, no convergence result seems attainable if $\bm{F}\not=0$ and, in the case $\bm{F}=0$, only weak convergence can be obtained in general. 
Instead, a modified approximation must be considered that uses a mass-lumping operator for the reaction term. 

Specifically, let $V_{\mathfrak{T},0}$ be a subspace of $H^1_{\mathfrak{T},0}$; we select a basis $(\chi_i)_{i\in I}$ of $V_{\mathfrak{T},0}$ and disjoint subsets $(U_i)_{i\in I}$ of $\Omega$, and we define the mass-lumping operator $\Pi_{\mathfrak{T}}: V_{\mathfrak{T},0}\to L^\infty(\Omega)$ by:
\begin{equation}\label{eq:def.ml.nc}
\forall v=\sum_{i\in I}v_i \chi_i\,,\quad\Pi_{\mathfrak{T}} v = \sum_{i\in I}v_i \mathbf{1}_{U_i},
\end{equation}
where $\mathbf{1}_{U_i}(\mathbi{x})=1$ if $\mathbi{x}\in U_i$ and $\mathbf{1}_{U_i}(\mathbi{x})=0$ otherwise. 
Note that the design of $\Pi_{\mathfrak{T}}$ actually depends on $V_{\mathfrak{T},0}$, and not just on the polytopal mesh $\mathfrak{T}$, but the natural notation $\Pi_{V_{\mathfrak{T},0}}$ has been simplified to $\Pi_{\mathfrak{T}}$ for legibility.

The function $\Pi_{\mathfrak{T}} v$ is piecewise constant and can be considered a good substitute of $v$, provided that each $v_i$ represents some approximate value of $v$ on $U_i$. In this setting, it also makes sense to define $\zeta(v)\in V_{\mathfrak{T},0}$ by applying the non-linear function $\zeta$ component-wise:
\[
\forall v=\sum_{i\in I}v_i \chi_i\,,\quad\zeta(v)=\sum_{i\in I}\zeta(v_i)\chi_i.
\]

\begin{remark}[Mass-lumping of the non-conforming {$\mathbb{P}^1$} method]
Let us illustrate the mass-lumping process on the non-conforming $\mathbb{P}^1$ method on a simplicial mesh. 
A basis of its space is given by $(\chi_{\sigma})_{\sigma\in\mathcal F_{\rm int}}$, where each $\chi_\sigma$ is piecewise linear in each element, with value $1$ at the centre of $\sigma$ and $0$ at the centres of all other faces. 
A mass-lumping operator $\Pi_{\mathfrak{T}}$ for this method is constructed in the following way: for each $v=\sum_{\sigma\in\mathcal F_{\rm int}}v_\sigma\chi_\sigma$,  let $\Pi_{\mathfrak{T}} v$ be the piecewise constant function equal to $v_\sigma$ on each diamond $D_\sigma$, $\sigma\in\mathcal F_{\rm int}$,   (and $\Pi_{\mathfrak{T}} v=0$ on the half-diamonds around boundary faces), see Fig. \ref{fig-mlP1} for an illustration.

\begin{figure}
\begin{center}
\begin{tabular}{c@{\qquad}c}
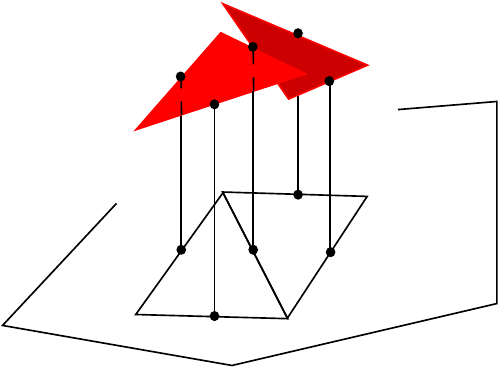 & 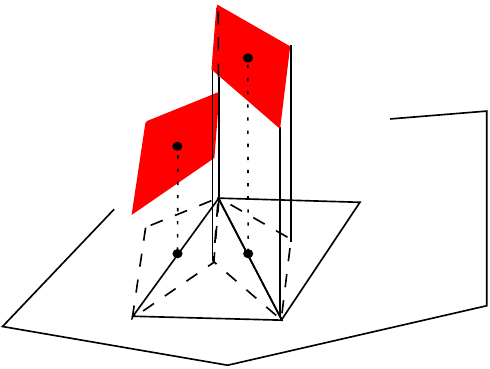
\end{tabular}
\caption{Example of a non-conforming $\mathbb{P}^1$ function (left) and its mass-lumped version (right).}
\label{fig-mlP1}\end{center}
\end{figure}
\end{remark}

A non-conforming approximation of \eqref{eq:stefan.weak} is then obtained replacing $H^1_0(\Omega)$ by $V_{\mathfrak{T},0}$, $\nabla$ with $\nabla_{\mathcal M}$ and using $\Pi_V$ in the reaction and source terms:
\begin{equation}\begin{array}{l}
\mbox{Find $u \in V_{\mathfrak{T},0}$ such that, }\forall v \in V_{\mathfrak{T},0},\\
\displaystyle\int_\Omega\left(\Pi_{\mathfrak{T}} u\,\Pi_{\mathfrak{T}} v + \Lambda\nabla_{\mathcal M}\zeta(u)\cdot\nabla_{\mathcal M} v\right) d\mathbi{x}  
= \int_\Omega f\Pi_{\mathfrak{T}} v d\mathbi{x} 
-\int_\Omega \bm{F}\cdot\nabla_{\mathcal M} v d\mathbi{x}.
\end{array}\label{eq:stefan.nc}\end{equation}

\begin{remark}[Computing the source and reaction terms]\label{rem:shape.ml.regions}
In practice, the right-hand side in \eqref{eq:stefan.nc} is never computed exactly, but through a low order quadrature rule on $f$, assuming that $f$ is approximated by a piecewise constant function on each $U_i$. If $f$ is continuous, for example, one can take
\[
\int_\Omega f\Pi_{\mathfrak{T}}v d\mathbi{x}\approx \sum_{i\in I}|U_i|f(\mathbi{x}_i) v_i
\]
where $\mathbi{x}_i$ is a point selected in or close to $U_i$. The reaction term in \eqref{eq:stefan.nc} is trivial to (exactly) compute:
\[
\int_\Omega \Pi_{\mathfrak{T}}u \Pi_{\mathfrak{T}}vd\mathbi{x} =\sum_{i\in I} |U_i|u_iv_i.
\]
The matrix associated with this term in the scheme is therefore diagonal, as expected. These considerations show that only the measures of $(U_i)_{i\in I}$ are actually needed to implement \eqref{eq:stefan.nc}.
\end{remark}

The following convergence theorem results from the analysis in Section \ref{sec:prop.NC} -- see Theorems \ref{th:gdm.cv.stefan} and \ref{th:prop.generic.nc} together with Lemma \ref{GDM:ml.is.ok}. 
Error estimates could also be stated, but they are more complicated to present and require stronger assumptions on the solution to the Stefan equation; we therefore refer the interested reader to \cite{DE19} for details, in which a partial uniqueness result is also stated for the solution of \eqref{eq:stefan.nc}.
We also mention in passing that error estimates for transient Stefan/porous medium equations are established in \cite{CDGGBP20}; these estimates are stated in the generic framework of the Gradient Discretisation Method, which covers polytopal non-conforming methods. 
\begin{theorem}[Convergence of polytopal non-conforming methods for Stefan]\label{th:cv.stefan}~\\
Let $\gamma>0$ be a fixed number, and let $(\mathfrak{T}_m)_{m\in\mathbb{N}}$ be a sequence of polytopal meshes such that $\gamma_{\mathfrak T_m}\le \gamma$ for all $m\in\mathbb{N}$ and such that $h_{\mathcal M_m}\to 0$ as $m\to\infty$. For each $m\in\mathbb{N}$, take a finite-dimensional subspace $V_{\mathfrak{T}_m,0}$ of $H^1_{\mathfrak{T}_m,0}$ and a mass-lumping operator $\Pi_{\mathfrak T_m}:V_{\mathfrak{T}_m,0}\to L^\infty(\Omega)$ as in \eqref{eq:def.ml.nc}, and assume the following:
\begin{align}\label{eq:consistency.VT}
&\min_{v\in V_{\mathfrak{T}_m,0}}\norm{\phi-v}{H^1_{\mathfrak T,0}}\to 0\mbox{ as $m\to\infty$,} \quad \forall \phi\in H^1_0(\Omega),\\
\label{eq:Pi.comparison}
&\max_{v\in V_{\mathfrak T_m,0}\backslash\{0\}}\frac{\norm{v-\Pi_{\mathfrak T_m}v}{L^2(\Omega)}}{\norm{\nabla_{\mathcal M_m}v}{L^2(\Omega)^d}}\to 0\mbox{ as $m\to\infty$}.
\end{align}
Then, for all $m\in\mathbb{N}$ there exists $u_m\in V_{\mathfrak{T}_m,0}$ solution of \eqref{eq:stefan.nc} and, as $m\to\infty$, $\Pi_{\mathfrak{T}_m}\zeta(u_m)\to \zeta(\bar u)$ strongly in $L^2(\Omega)$, $\nabla_{\mathcal M_m}\zeta(u_m)\to\nabla\zeta(\bar u)$ strongly in $L^2(\Omega)^d$, and $\Pi_{\mathfrak{T}_m}u_m\to\bar u$ weakly in $L^2(\Omega)$, where $\bar u$ is a solution to \eqref{eq:stefan.weak}.
\end{theorem}

\section{A locally enriched polytopal non-conforming finite element scheme}\label{sec:bpncfe}

We describe here a non-conforming method that can be applied to almost any polytopal mesh as per Definition \ref{def:polymesh}. Actually, the only additional assumption we make on the mesh is the following:
\begin{equation} \label{conv:faces}
\forall \sigma\in\mathcal F\,,\;\mbox{$\sigma$ is convex}.
\end{equation}
This convexity assumption on the face is rather weak, and the cells themselves can be non-convex -- which is often the case in 3D. 

\medskip

Let us first describe the underlying idea. 
To ensure the consistency of the method, a basic requirement would be for the local spaces (restriction of $V_{\mathfrak{T},0}$ to a cell $K\in\mathcal M$) to contain $\mathbb{P}^1(K)$. 
Denoting by $\mathbb{P}^1(\mathcal M)$ the space of piecewise linear functions on the mesh, without continuity conditions, this means that we should have $\mathbb{P}^1(\mathcal M)\cap H^1_{\mathfrak{T},0}\subset V_{\mathfrak{T},0}$. 
This suggests to take $\mathbb{P}^1(\mathcal M)\cap H^1_{\mathfrak{T},0}$ as our non-conforming finite-dimensional space. 
However, if the number of faces of most of the elements is greater than $d+1$, the constraints of continuity at the faces will impede a correct interpolation.
For instance, on a domain $\Omega$ that can be meshed by uniform hexagons (see Fig. \ref{fig:polygons}), the space $\mathbb{P}^1(\mathcal M)\cap H^1_{\mathfrak{T},0}$ is reduced to $\{0\}$. 
Indeed, the three boundary conditions on the exterior edges of element 1 imply that the constant gradient vanishes in element 1. 
Therefore the mean values at the three interior edges of element 1 also vanish, so that the same reasoning holds in element 2. 
By induction, the gradient vanishes in all the elements of the mesh.
\begin{figure}[htb]
\begin{center}
	\begin{tikzpicture}[hexa/.style= {shape=regular polygon,regular polygon sides=6,minimum size=0.5cm, draw,inner sep=0,anchor=south,rotate=30}]
\foreach \j in {0,...,5}{%
\pgfmathsetmacro\end{5+\j} 
  \foreach \i in {0,...,\end}{%
  \node[hexa] (h\i;\j) at ({(\i-\j/2)*sin(60)*.5},{\j*0.75*.5}) {};}  }      
\foreach \j in {0,...,4}{%
  \pgfmathsetmacro\end{9-\j} 
  \foreach \i in {0,...,\end}{%
  \pgfmathtruncatemacro\k{\j+6}  
  \node[hexa] (h\i;\k) at ({(\i+\j/2-2)*sin(60)*.5},{(4.5+\j*0.75)*.5}) {};}  } 

  \foreach \k in {0,...,10}  {\pgfmathtruncatemacro{\kn}{\k+12}\node [circle,minimum size=0.5cm] at (h1;\k) {\kn};} 
   \foreach \k in {0,...,10}  {\pgfmathtruncatemacro{\kn}{\k+1}\node [circle,minimum size=0.5cm] at (h0;\k) {\kn};}   
\end{tikzpicture}
\caption{Hexagonal mesh\label{fig:polygons}}
\end{center}
\end{figure}
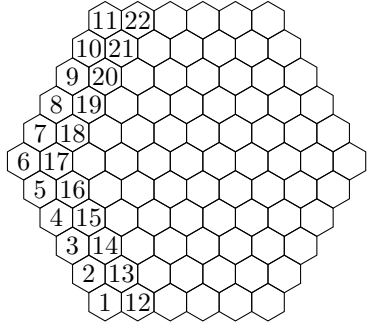

We therefore enrich this initial space with functions associated with the faces, that we use to ensure the proper continuity conditions by ``localising'' the basis of $\mathbb{P}^1$ inside each element. 
The resulting global basis is made of functions associated with the faces and of additional local functions on the cell.
 As a consequence, we call the corresponding method the Locally Enriched Polytopal Non-Conforming finite element method (LEPNC for short).

\begin{remark}[Link with the non conforming {$\mathbb{P}^1$} finite element method]~\\
Note that, when applied to a triangular mesh in 2D, the LEPNC yields 6 degrees of freedom on each triangle, while the classical non conforming $\mathbb{P}^1$ finite element (NCP1FE) method  has only 3.
However, when performing static condensation (see Remark \ref{rem:statcond}) on the LEPNC scheme on triangles, only the 3 degrees of freedom pertaining to the faces remain, so that the computational cost is close to that of the NCP1FE scheme.
In fact, the precision of the methods are close. 
Morever, in the case of an elliptic equation with non homogeneous Dirichlet boundary conditions and a zero right hand side, the approximate solutions given by the NCP1FE and the condensed LEPNC schemes are identical. 
\end{remark}

\subsection{Local space}\label{sec:const.ncPoly}

We first describe the local spaces and shape functions. 
Let $K\in\mathcal M$, for $\sigma\in\mathcal F_K$, the pyramid $D_{K,\sigma}$ has $\sigma$ as one of its faces, as well as faces $\tau$ that are internal to $K$, and gathered in the set $\mathcal F_{K\sigma,{\rm int}}$; see Fig. \ref{fig:pyramid} for an illustration. 

\begin{figure}
\centerline{
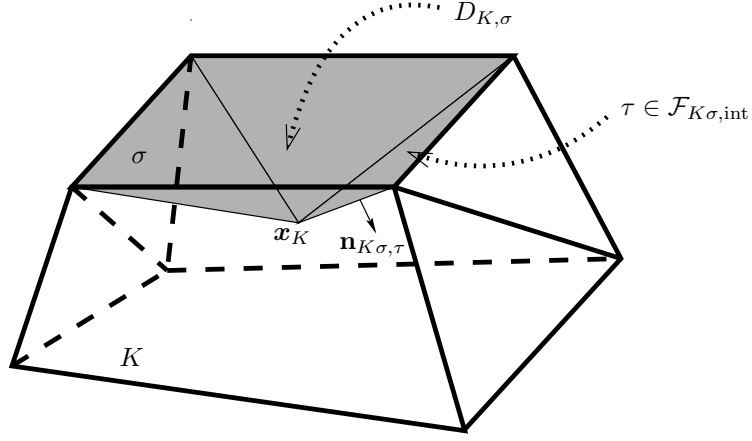
}
\caption{Notations for the design of the local polytopal non-conforming space of Section \ref{sec:const.ncPoly}}
\label{fig:pyramid}
\end{figure}

Let $\phi_{K,\sigma}:K\to\mathbb{R}$ be the piecewise-polynomial function such that, inside $D_{K,\sigma}$, $\phi_{K,\sigma}$ is the product of the distances to each internal face $\tau\in\mathcal F_{K\sigma,{\rm int}}$, and outside $D_{K,\sigma}$ we set $\phi_{K,\sigma}=0$. 
Additionally, $\phi_{K,\sigma}$ is scaled in order to have an average equal to one on $\sigma$. 
The function $\phi_{K,\sigma}$ vanishes on all the faces of $D_{K,\sigma}$ except $\sigma$. 
Under the convexity assumption \eqref{conv:faces} and letting $\mathbf{n}_{K\sigma,\tau}$ be the outer unit normal to $D_{K,\sigma}$ on $\tau\in\mathcal F_{K\sigma,{\rm int}}$, we therefore set
\begin{equation}\label{NCpoly:def.phiKsigma}
\phi_{K,\sigma}(\mathbi{x})=c_{K,\sigma}\prod_{\tau\in\mathcal F_{K\sigma,{\rm int}}} \left[(\mathbi{x}_K-\mathbi{x})\cdot\mathbf{n}_{K\sigma,\tau} \right]^+\quad\forall\mathbi{x}\in K,
\end{equation}
where $s^+=\max(s,0)$ is the positive part of $s\in\mathbb{R}$. 
As previously mentioned, $c_{K,\sigma}>0$ is chosen to ensure that $\phi_{K,\sigma}$ has an average of one on $\sigma$; since this function vanishes outside $D_{K,\sigma}$, this means that we have
\begin{equation}\label{NCpoly:normalisation}
\frac{1}{|\sigma|}\int_\sigma \phi_{K,\sigma}=1\,,\mbox{ and }
\int_{\sigma'} \phi_{K,\sigma}=0\quad\forall \sigma'\in\mathcal F_K\backslash\{\sigma\}.
\end{equation}
We then define the local space on $K$ of the LEPNC method by
\begin{equation}\label{eq:local.V}
V^{\textsc{\tiny LEPNC}}_K:={\rm span}(\mathbb{P}^1(K)\cup\{\phi_{K,\sigma}\,:\,\sigma\in\mathcal F_K\}).
\end{equation}
The component $\mathbb{P}^1(K)$ will be responsible for the approximation properties of the global space, whereas the face-based basis functions will be used to glue local spaces together and ensure \eqref{ncfe:eq:sameav}.

\begin{remark}[Nature of the functions in the local space]
The functions of $V^{\textsc{\tiny LEPNC}}_K$ are continuous on $K$, and polynomial in each pyramid $D_{K,\sigma}$ for $\sigma\in\mathcal F_K$. 
The maximal polynomial degree of functions in $V^{\textsc{\tiny LEPNC}}_K$ is $\max_{\sigma\in\mathcal F_K}{\rm Card}(\mathcal E_\sigma)$, where $\mathcal E_\sigma$ is the set of edges of $\sigma$ (vertices in 2D, in which case the maximal degree is 2).

A practical implementation of any non-conforming method requires to integrate the local functions and their gradients on each cell. 
For $V^{\textsc{\tiny LEPNC}}_K$, this is very easy: one simply has to select quadrature rules in $K$ that are constructed by assembling quadrature rules on each pyramid. 
This is actually a standard way of constructing quadrature rules on polytopal cells, these pyramids being then cut into tetrahedra on which quadrature rules are known.
\end{remark}

\subsection{Global LEPNC space and basis of functions}\label{NCpoly:sec:spaceV}

The global non-conforming space of the Locally Enriched Polytopal Non-Conforming method is 
\begin{equation}\label{def:V.T}
V^{\textsc{\tiny LEPNC}}_{\mathfrak{T},0}=\{v\in H^1_{\mathfrak T,0}\,:\,v_{|K}\in V^{\textsc{\tiny LEPNC}}_K\quad\forall K\in\mathcal M\}.
\end{equation}
By construction of $(V^{\textsc{\tiny LEPNC}}_K)_{K\in\mathcal M}$, an explicit and local basis of $V^{\textsc{\tiny LEPNC}}_{\mathfrak{T},0}$ can be constructed thanks to the functions $(\phi_{K,\sigma})_{K\in\mathcal M\,,\;\sigma\in\mathcal F_K}$.
For each $\sigma\in\mathcal F$, first define the function $\phi_\sigma:\Omega\to\mathbb{R}$ by patching the local functions, in the cells on each side of $\sigma$, associated with $\sigma$:
\begin{equation}\label{NCpoly:def.phisigma}
(\phi_\sigma)_{|K}=\phi_{K,\sigma}\quad\forall K\in\mathcal M_\sigma\,,\qquad (\phi_\sigma)_{|L}=0\mbox{ if $L\not\in\mathcal M_\sigma$}.
\end{equation}
The properties \eqref{NCpoly:normalisation} ensure that $\phi_\sigma$ satisfies 1. and 2. in Definition \ref{chapncfe:def:ncspace} (it also satisfies 3. if $\sigma\in\mathcal F_{\rm int}$). 
We also note that each $\phi_\sigma$ is a sort of bubble function on the diamond $D_\sigma$, as it vanishes on all its faces (but, contrary to standard bubble functions, $\phi_\sigma$ is not in $H^1(D_\sigma)$).

We then select, for each $K\in\mathcal M$, $d+1$ vertices $(\mathbi{s}_0,\ldots,\mathbi{s}_d)$ of $K$ which maximise the volume of their convex hull, that is, maximise their determinant; in fact the determinant only needs to be non-zero, but maximising it leads to better conditioned matrices. 
We then define the nodal basis $(\psi_{K,i})_{i=0,\ldots,d}$ of $\mathbb{P}^1(K)$ associated to these vertices, that is, the basis that satisfies $\psi_{K,i}(\mathbi{s}_j) = 1$ if $i=j$ and $0$ if $i\neq j$. 
We will see in Section \ref{NCpoly:sec:ml} that this choice is relevant for mass lumping techniques. For each $i=0,\ldots,d$, we set
\begin{equation}\label{NCpoly:def.phiKi}
\phi_{K,i}=\psi_{K,i}-\sum_{\sigma\in\mathcal F_K}\overline{\psi}_{K,i,\sigma}\phi_{K,\sigma}\quad\mbox{ with }\quad
\overline{\psi}_{K,i,\sigma}=\frac{1}{|\sigma|}\int_\sigma \psi_{K,i}.
\end{equation}
This choice ensures that
\begin{equation}\label{NCpoly:phiKi.zero}
\int_\sigma \phi_{K,i}=0\qquad\forall \sigma\in\mathcal F_K.
\end{equation}
Extended by 0 outside $K$, each $\phi_{K,i}$ therefore belongs to $H^1_{\mathfrak T,0}$. It can also easily be checked that $\{\phi_{K,i}\,:\,i=0,\ldots,d\}\cup\{\phi_{K,\sigma}\,:\,\sigma\in\mathcal F_K\}$ spans $V^{\textsc{\tiny LEPNC}}_K$ (the basis $(\psi_{K,i})_{i=0,\ldots,d}$ of $\mathbb{P}^1(K)$ can be obtained by linear combinations of these functions).
As shown in the following lemma, a basis of $V^{\textsc{\tiny LEPNC}}_{\mathfrak{T},0}$ is then obtained by gathering all the functions \eqref{NCpoly:def.phisigma} (for internal faces) and \eqref{NCpoly:def.phiKi}.

\begin{lemma}[Basis of the LEPNC global space]
\label{NCpoly:lem.basis}
The following family forms a basis of $V^{\textsc{\tiny LEPNC}}_{\mathfrak{T},0}$ defined by \eqref{def:V.T}:
\begin{equation}\label{NCpoly:basis.V}
\{\phi_{K,i}\,:\,K\in\mathcal M\,,\;i=0,\ldots,d\}\cup\{\phi_\sigma\,:\,\sigma\in\mathcal F_{\rm int}\}.
\end{equation}
 Moreover, for any $v\in V^{\textsc{\tiny LEPNC}}_{\mathfrak{T},0}$ we have
\begin{equation}\label{NCpoly:dec.v}
v=\sum_{K\in\mathcal M}\sum_{i=0}^d v_{K,i}\phi_{K,i}+\sum_{\sigma\in\mathcal F_{\rm int}}v_\sigma\phi_\sigma,
\end{equation}
with
\begin{equation}\label{NCpoly:eq.lambda.sigma}
v_\sigma=\frac{1}{|\sigma|}\int_\sigma v\qquad\forall\sigma\in\mathcal F_{\rm int}.
\end{equation}
and, for all $K\in \mathcal M$ ,
\begin{equation}\label{NCpoly:eq.vertex}
v_{K,i} = v_{|K}(\mathbi{s}_i)\quad\forall i=0,\ldots,d.
\end{equation}
\end{lemma}
 
\begin{remark}[Single-valuedness of $v_\sigma$]\label{rem:unique.vs}
We note that, since $v\in H^1_{\mathfrak{T},0}$, the condition \eqref{ncfe:eq:sameav} ensures that $v_\sigma$ is uniquely defined by \eqref{NCpoly:eq.lambda.sigma} (it depends only on $\sigma$, not on the choice of a cell in $\mathcal M_\sigma$ in which we would consider the values of $v$).
\end{remark}

\begin{proof}
Proving \eqref{NCpoly:dec.v}--\eqref{NCpoly:eq.vertex} for a generic $v\in V^{\textsc{\tiny LEPNC}}_{\mathfrak{T},0}$ shows that \eqref{NCpoly:basis.V} spans this space, and also that it is a linearly independent family since all coefficients in the right-hand side of \eqref{NCpoly:dec.v} vanish when the left-hand side $v$ vanishes.

Let us take $v\in V^{\textsc{\tiny LEPNC}}_{\mathfrak{T},0}$. It suffices to show that \eqref{NCpoly:dec.v} holds on each cell $K\in\mathcal M$. Since $\{\phi_{K,i}\,:\,i=0,\ldots,d\}\cup\{\phi_{K,\sigma}\,:\,\sigma\in\mathcal F_K\}$ spans $V^{\textsc{\tiny LEPNC}}_K \ni v_{|K}$, there are coefficients $(\lambda_{K,i})_{i=0,\ldots,d}$ and $(\lambda_{K,\sigma})_{\sigma\in\mathcal F_K}$ such that
\begin{equation}\label{eq:v.local.dec}
v_{|K}=\sum_{i=0}^d\lambda_{K,i}\phi_{K,i} + \sum_{\sigma\in\mathcal F_K}\lambda_{K,\sigma}\phi_\sigma.
\end{equation}
Taking the average over one face $\sigma\in\mathcal F_K$ and using \eqref{NCpoly:normalisation} and \eqref{NCpoly:phiKi.zero}, we obtain 
\[
\lambda_{K, \sigma}=\frac{1}{|\sigma|}\int_\sigma v_{|K}.
\]
Hence, by Remark \ref{rem:unique.vs}, $\lambda_{K,\sigma}=v_\sigma$ defined by \eqref{NCpoly:eq.lambda.sigma}. 
Applying now \eqref{eq:v.local.dec} at one of the vertices $\mathbi{s}_i$, recalling the definition \eqref{NCpoly:def.phiKi}, the fact that $(\psi_{K,j})_{j=0,\ldots,d}$ is the nodal basis associated with $(\mathbi{s}_j)_{j=0,\ldots,d}$, and noticing that all functions $\phi_{K,\sigma}$ vanish at the vertices of $K$ (consequence of \eqref{NCpoly:def.phiKsigma} and of the fact that each vertex either does not belong to $D_{K,\sigma}$, or belongs to one face in $\mathcal F_{K\sigma, {\rm int}}$), we see that $v_{|K}(\mathbi{s}_i)=\lambda_{K,i}$. To summarise, \eqref{eq:v.local.dec} is written
\begin{equation}\label{eq:v.local.dec2}
v_{|K}=\sum_{i=0}^dv_{K,i}\phi_{K,i} + \sum_{\sigma\in\mathcal F_K\cap\mathcal F_{\rm int}}v_{\sigma}\phi_\sigma,
\end{equation}
the restriction of the last sum to internal edges coming from $\int_\sigma v=0$ whenever $\sigma\in\mathcal F_{\rm ext}$, see \eqref{ncfe:eq:avnull}. Since all functions $\phi_{L,i}$ vanish on $K$ whenever $L\not=K$, and all $\phi_\sigma$ vanish on $K$ whenever $\sigma\not\in \mathcal F_K$, \eqref{eq:v.local.dec2} proves that \eqref{NCpoly:dec.v} holds on $K$.
\end{proof}

Let $C(\mathcal M)$ denote the functions whose restriction to each $K\in\mathcal M$ is continuous on $\overline{K}$. 
Lemma \ref{NCpoly:lem.basis} shows us how to define a natural interpolator $\mathcal I_{\mathfrak{T}}:H^1(\Omega)\cap C(\mathcal{M})\to V^{\textsc{\tiny LEPNC}}_{\mathfrak{T},0}$: for all $u\in H^1(\Omega)\cap C(\mathcal{M})$:
\begin{subequations}\label{NCpoly:def:ID}
\begin{equation}\label{NCpoly:def:ID.expr}
\mathcal I_{\mathfrak{T}}u=\sum_{K\in\mathcal M}\sum_{i=0}^d u_{K,i}\phi_{K,i}+\sum_{\sigma\in\mathcal F_{\rm int}}u_\sigma\phi_\sigma
\end{equation}
where $(u_\sigma)_{\sigma\in\mathcal F_{\rm int}}$ and $(u_{K,i})_{K\in\mathcal M,\,i=0,\cdots,d}$ are defined by
\begin{align}
\label{NCpoly:interpolator.sigma}
&u_\sigma=\frac{1}{|\sigma|}\int_\sigma u\quad\forall \sigma\in\mathcal F_{\rm int}\,,\\
\label{NCpoly:interpolator.K}
&u_{K,i} = u_{|K}(\mathbi{s}_i)\quad\forall K\in\mathcal M\,,\;\forall i=0,\ldots,d.
\end{align}
\end{subequations}

\begin{remark}[The need to enrich the bubble functions]\label{rem:need.bubble}
As shown by the above construction (see in particular \eqref{NCpoly:def.phiKi}), the design of a finite-dimensional subspace of the non-conforming space $H^1_{\mathfrak T,0}$ requires access, for each face $\sigma$ of each cell $K$, to a local basis function that has average 1 on $\sigma$ and 0 on all other faces of $K$. 
Instead of using the bubble functions \eqref{NCpoly:def.phiKsigma}, an alternative idea is to use a rich enough space of polynomial functions. 
The question of ``how rich'' this space should be (which degree the polynomials should have) is however not easy to answer, when considering generic polytopal meshes.

\begin{figure}
\begin{center}
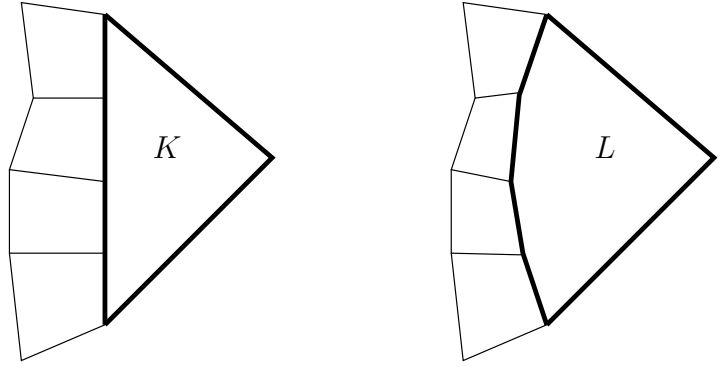
\end{center}
\caption{Hexagons with aligned (left) and almost aligned (right) edges.
\label{fig:aligned}
}
\end{figure}

Consider for example the cell $K$ on the left of Fig. \ref{fig:aligned}, an hexagon with 4 aligned edges. Since it has a total of 6 edges, the minimum local space of polynomial should be $\mathbb{P}^2(K)$, which has dimension 6. 
However, the restrictions of functions in $\mathbb{P}^2(K)$ on the line of the aligned edges are polynomials of degree 2 in dimension 1, and form therefore a space of dimension 3. This space is not large enough to contain, for each of the 4 edges, a function with  average 1 on this edge and 0 on all other edges. 
This shows that we should at least consider $\mathbb{P}^3(K)$ as the local polynomial space on $K$; note that this argument only discusses the space dimension: it would still have to be fully established that $\mathbb{P}^3(K)$ is indeed rich enough.

The situation is perhaps more severe, from the robustness point of view, for the hexagon $L$ on the right of Fig. \ref{fig:aligned}. 
Since its edges are not aligned, from the pure dimensional point of view it might be sufficient to consider $\mathbb{P}^2(L)$ as the local polynomial space on $L$. 
However, because $L$ has \emph{almost aligned} edges, the basis functions we would construct (with average 1 on one edge and 0 on all other edges) would form an ``almost dependent'' set of functions -- even more so as the edges become more and more aligned, e.g.\ along a sequence of refined meshes. 
The practical consequence is that, in an implementation of the scheme using these basis functions, some local mass or stiffness matrices would be close to singular, which would lead to an ill-conditioned global system and a poor numerical resolution.

On the contrary, the usage of the (piecewise-polynomial) basis functions \eqref{NCpoly:def.phiKsigma} solves these two issues: the local space is always defined as the span of $\mathbb{P}^1$ and the bubble functions, independently of the cell geometry, and, even when edges become aligned, the basis functions remain well independent (recall that the vertices $(\mathbf{s}_0,\ldots,\mathbf{s}_d)$ are chosen in each cell to maximise the volume they encompass and thus, in Fig. \ref{fig:aligned}, they would be chosen as the three leftmost vertices in each case and would not become aligned or close to aligned).
\end{remark}

\subsection{Approximation properties of the LEPNC space}\label{sec:appprobpnc}

The approximation properties of the LEPNC space require a slightly more stringent, but still very flexible, regularity condition on the meshes than the boundedness of $\gamma_{\mathfrak T}$ (see \eqref{def:reg.gamma}).

\begin{definition}[$\rho$-regular polytope and polytopal mesh]\label{def:rhoregpoly}
A polytopal open set $K\subset\mathbb{R}^d$ is said to be a $\rho$-regular polytope, where $\rho>0$, if:
\begin{enumerate}
 \item There exists $\mathbi{x} _{K}\in K$ and open disjoint simplices $(K_i)_{i=1,\ldots,n}$ such that $\overline{K} = \bigcup_{i=1}^n \overline{K}_i$,  and, for $i=1,\ldots,n$, $\mathbi{x} _{K}$ is a vertex of $K_i$, exactly one face of $K_i$ is included in $\partial K$ and all the other faces of $K_i$ are common with a neighbouring simplex $K_j$.
 \item There exists $\mathbi{x} _{K_i}\in K_i$ such that $B(\mathbi{x} _{K_i},\rho h_K)\subset K_i$.
\end{enumerate}
A $\rho$-regular polytopal mesh of $\Omega$ is a polytopal mesh $\mathfrak{T}$ as per Definition \ref{def:polymesh}, such that any cell $K\in\mathcal{M}$ is a $\rho$-regular polytope and if, for any simplex $K_i$ as above, there exists $\sigma\in\mathcal F_K$ such that one face of $K_i$ is included in $\sigma$.
\end{definition}

\begin{remark}[$\rho$-regular polytope and polytopal mesh]\label{rem:rho.poly.nb.simplices}
The number $n$ in  Definition \ref{def:rhoregpoly} is always bounded by $1/\rho^d$, the ratio of the measure of $B(\mathbi{x} _K,h_K)$ and that of $B(\mathbi{x} _{K_i},\rho h_K)$. As a consequence, it can be easily checked that $\gamma_{\mathfrak T}$ (defined by \eqref{def:reg.gamma}) is bounded above by a real number depending only on $\rho$.

The additional requirement, for a polytopal mesh, that one face of $K_i$ is included in one of the mesh face prevents the situation where the face of $K_i$ that lies in $\partial K$ is actually split between two mesh faces (the mesh faces could be different from the geometrical faces of its elements, e.g. in case of non-conforming meshes with hanging nodes).
\end{remark}
To state approximation properties of the global non-conforming space \eqref{def:V.T}, we first define an alternate interpolator, which does not require the functions to be continuous on each cell and therefore enjoys boundedness properties for a larger class of functions. For all $K\in\mathcal{M}$, let $\mathcal J_K: H^1(K)\to V^{\textsc{\tiny LEPNC}}_K$ be such that
 \begin{equation}\label{NCpoly:defJK}
 {\mathcal J}_{K}u=\mathcal J_{\mathcal F_K}u + P_K (u - \mathcal J_{\mathcal F_K}u)\quad\forall u\in H^1(K),
 \end{equation}
 where
 \begin{equation}\label{NCpoly:defJK.bdry}
 \mathcal J_{\mathcal F_K}u=\sum_{\sigma\in\mathcal F_K}u_\sigma\phi_{K,\sigma}\quad\mbox{ with $(u_\sigma)_{\sigma\in\mathcal F_K}$ given by \eqref{NCpoly:interpolator.sigma}},
 \end{equation}
 and $P_K:L^2(K)\to V^{\textsc{\tiny LEPNC}}_K$ is the $L^2$-orthogonal projector on ${\rm span}\{\phi_{K,i}\,:\,i=0,\ldots,d\}$. The global interpolator $\mathcal J_{\mathfrak{T}}:H^1_0(\Omega)\to V^{\textsc{\tiny LEPNC}}_{\mathfrak{T},0}$ is obtained patching the local ones:
 \[
 (\mathcal J_{\mathfrak{T}}u)_{|K}=\mathcal J_K (u_{|K})\quad\forall u\in H^1_0(\Omega)\,,\;\forall K\in\mathcal{M}.
 \]
 Using \eqref{NCpoly:normalisation} and \eqref{NCpoly:phiKi.zero}, it is easily verified that $\mathcal J_{\mathfrak{T}}u$ indeed belongs to $V^{\textsc{\tiny LEPNC}}_{\mathfrak{T},0}$. 

 \begin{theorem}[Approximation properties of {$V^{\textsc{\tiny LEPNC}}_{\mathfrak{T},0}$}]\label{NCpoly:th.approx}
Assume that $\mathfrak T$ is a $\rho$-regular polytopal mesh. Then, there exists $C$ depending only on $\rho$ such that
\begin{equation}\label{NCpoly:eq.approx.JD}
\|u-\mathcal J_{\mathfrak{T}}u\|_{L^2(\Omega)} + h_{\mathcal M}\|\nabla_{\mathcal M}(u-\mathcal J_{\mathfrak{T}}u)\|_{L^2(\Omega)}
\le Ch_{\mathcal M}^2|u|_{H^2(\Omega)}\quad\forall u\in H^1_0(\Omega)\cap H^2(\Omega),
\end{equation}
where $|{\cdot}|_{H^2(\Omega)}$ denotes the $H^2(\Omega)$-seminorm.
\end{theorem}

\begin{remark}[Approximation properties in generic Sobolev spaces]
Using the results of \cite[Chapter 1]{hho-book}, a straightforward adaptation of the proof below
shows that the approximation property \eqref{NCpoly:eq.approx.JD} also holds with $L^2$, $H^1_0$ and $H^2$ replaced by $L^p$, $W^{1,p}_0$ and $W^{2,p}$, for any $p\in [1,\infty)$.
\end{remark}

Before proving this theorem, let us estabish the boundedness of the local interpolator $\mathcal J_K$.

\begin{lemma}[Boundedness of {$\mathcal J_K$}]\label{NCpoly:lem.bound.interp}
Assume that $K$ is a $\rho$-regular polytope. Then, there exists $C>0$ depending only on $\rho$ such that, for all $u\in H^1(K)$,
\begin{align}\label{NCpoly:eq.bound.JK}
\|\mathcal J_K u\|_{L^2(K)}\le{}& C (\|u\|_{L^2(K)}+h_K\|\nabla u\|_{L^2(K)^d})\,,\\
\label{NCpoly:eq.bound.JK.grad}
\|\nabla \mathcal J_K u\|_{L^2(K)^d}\le{}& C \|\nabla u\|_{L^2(K)^d}.
\end{align}
\end{lemma}

\begin{proof}

In this proof, $C>0$ denotes a generic real number, that can change from one line to the next but depends only on $\rho$.

\textbf{Step 1}: \emph{Polynomial invariance of $\mathcal J_K$ and estimates on the basis functions.}

The definitions \eqref{NCpoly:def.phiKi} and \eqref{NCpoly:defJK.bdry} show that
$\phi_{K,i}=\psi_{K,i}-\mathcal J_{\mathcal F_K}\psi_{K,i}$ for all $i=0,\ldots,d$. Hence, $P_K(\psi_{K,i}-\mathcal J_{\mathcal F_K}\psi_{K,i})=P_K\phi_{K,i}=\phi_{K,i}$ and $\mathcal J_K \psi_{K,i}=\mathcal J_{\mathcal F_K}\psi_{K,i}+\phi_{K,i}=\psi_{K,i}$.
Since $\mathbb{P}^1(K)={\rm span}\{\psi_{K,i}\,:\,i=0,\ldots,d\}$ this establishes the following polynomial invariance of $\mathcal J_K$:
\begin{equation}\label{NCpoly:JK.poly}
\mathcal J_K q=q\quad\forall q\in\mathbb{P}^1(K).
\end{equation}

The definition \eqref{NCpoly:def.phiKsigma} and the $\rho$-regularity of $K$ imply that $\phi_{K,\sigma}\ge  c_{K,\sigma} Ch_\sigma^{n_\sigma}$ on a ball $B_\sigma$ in $\sigma$ of diameter $C h_\sigma$, where $h_\sigma$ is the diameter of $\sigma$ and $n_\sigma={\rm Card}(\mathcal F_{K\sigma,{\rm int}})$. Integrating this relation over $B_\sigma$, using \eqref{NCpoly:normalisation} and noticing that $|\sigma|\le C|B_\sigma|$, we infer $c_{K,\sigma}\le C h_\sigma^{-n_\sigma}$ and thus, since $h_K\le Ch_\sigma$ by $\rho$-regularity of $K$,
\begin{equation}\label{NCpoly:bound.phiKs}
|\phi_{K,\sigma}|\le C\quad\mbox{ on $K$.}
\end{equation}
The same definition \eqref{NCpoly:def.phiKsigma} also yields $|\nabla\phi_{K,\sigma}|\le c_{K,\sigma}C h_K^{n_\sigma-1}$ on $K$, and therefore
\begin{equation}\label{NCpoly:bound.nablaPhiKs}
|\nabla\phi_{K,\sigma}|\le Ch_K^{-1}\quad\mbox{ on $K$.}
\end{equation}

\textbf{Step 2}: \emph{Estimate on $\nabla\mathcal J_K u$.}

By \eqref{NCpoly:JK.poly}, $\mathcal J_K 1 = 1$ and thus $\nabla\mathcal J_K u=\nabla\mathcal J_K(u-\overline{u}_K)$, where
$\overline{u}_K=\frac{1}{|K|}\int_K u$, which implies
\begin{equation}\label{NCpoly:stab.JK}
\nabla\mathcal J_K u = \nabla \mathcal J_{\mathcal F_K}(u-\overline{u}_K)+\nabla P_K[(u-\overline{u}_K)-\mathcal J_{\mathcal F_K}(u-\overline{u}_K)].
\end{equation}
Let us first estimate $\nabla \mathcal J_{\mathcal F_K}(u-\overline{u}_K)$. By \cite[Est. (B.11)]{gdm} we have
\[
|u_\sigma-u_K|^2
\le 
\frac{Ch_K}{|\sigma|}\int_{K}|\nabla u|^2d\mathbi{x} \quad
\forall \sigma\in\mathcal F_K,
\]
from which we deduce
\[
|\nabla \mathcal J_{\mathcal F_K}(u-\overline{u}_K)|\le C\sum_{\sigma\in\mathcal F_K}\frac{h_K}{(|\sigma|h_K)^{1/2}}\|\nabla u\|_{L^2(K)^d}\,|\nabla\phi_{K,\sigma}|.
\]
The estimate \eqref{NCpoly:bound.nablaPhiKs} yields $\|\nabla\phi_{K,\sigma}\|_{L^2(K)^d}\le Ch_K^{-1}|K|^{1/2}$ and thus, since $|K|\le C|\sigma|h_K$ and ${\rm Card}(\mathcal F_K)\le C$ (consequence of Remark \ref{rem:rho.poly.nb.simplices}),
\begin{equation}
\|\nabla \mathcal J_{\mathcal F_K}(u-\overline{u}_K)\|_{L^2(K)^d}\le C \,\|\nabla u\|_{L^2(K)^d}.
\label{NCpoly:stab.JFK.grad}
\end{equation}
The same arguments with $\phi_{K,\sigma}$ instead of $\nabla\phi_{K,\sigma}$ and \eqref{NCpoly:bound.phiKs} instead of \eqref{NCpoly:bound.nablaPhiKs} yields
\begin{equation}
\|\mathcal J_{\mathcal F_K}(u-\overline{u}_K)\|_{L^2(K)}\le C h_K\,\|\nabla u\|_{L^2(K)^d}.
\label{NCpoly:stab.JFK}
\end{equation}

We now turn to the second term in the right-hand side of \eqref{NCpoly:stab.JK}. 
The range of $P_K$ is contained in a space of piecewise polynomials, with uniformly bounded degree, on a regular subdivision of $K$. The inverse inequality of \cite[Lemma 1.28 and Remark 1.33]{hho-book} therefore gives
\[
\|\nabla P_K[(u-\overline{u}_K)-\mathcal J_{\mathcal F_K}(u-\overline{u}_K)]\|_{L^2(K)^d}\le C h_K^{-1}\|P_K[(u-\overline{u}_K)-\mathcal J_{\mathcal F_K}(u-\overline{u}_K)]\|_{L^2(K)}.
\]
Since $P_K$ is an $L^2$-orthogonal projection, we infer
\begin{align}
\|\nabla P_K[(u-\overline{u}_K)-\mathcal J_{\mathcal F_K}(u-\overline{u}_K)]\|_{L^2(K)^d}\le{}& C h_K^{-1}\|(u-\overline{u}_K)-\mathcal J_{\mathcal F_K}(u-\overline{u}_K)\|_{L^2(K)}\nonumber\\
\le{}& Ch_K^{-1}\|u-\overline{u}_K\|_{L^2(K)} + Ch_K^{-1}\|\mathcal J_{\mathcal F_K}(u-\overline{u}_K)\|_{L^2(K)}\nonumber\\
\le{}& C \|\nabla u\|_{L^2(K)^d},
\label{eq:est.nabla.PK}
\end{align}
where we have used $\norm{u-u_K}{L^2(K)}\le Ch_K\norm{\nabla u}{L^2(K)^d}$ (see \cite[Est. (B.12)]{gdm}) and \eqref{NCpoly:stab.JFK} in the last line. Combined with \eqref{NCpoly:stab.JFK.grad} and \eqref{NCpoly:stab.JK}, this proves \eqref{NCpoly:eq.bound.JK.grad}.

\medskip

\textbf{Step 3}: \emph{Estimate on $\mathcal J_K u$.}

We use the triangle inequality together with $\mathcal J_K \overline{u}_K=\overline{u}_K$ (see \eqref{NCpoly:JK.poly}) to write
\begin{align*}
\|\mathcal J_K u\|_{L^2(K)}\le{}& \|\mathcal J_K (u-\overline{u}_K)\|_{L^2(K)}+\|\overline{u}_K\|_{L^2(K)}\\
\le{}& \|\mathcal J_{\mathcal F_K} (u-\overline{u}_K)\|_{L^2(K)} + \|P_K[(u-\overline{u}_K)-\mathcal J_{\mathcal F_K} (u-\overline{u}_K)]\|_{L^2(K)}+\|u\|_{L^2(K)}\\
\le{}& \|\mathcal J_{\mathcal F_K} (u-\overline{u}_K)\|_{L^2(K)} + \|(u-\overline{u}_K)-\mathcal J_{\mathcal F_K} (u-\overline{u}_K)\|_{L^2(K)}+\|u\|_{L^2(K)}\\
\le{}& Ch_K\|\nabla u\|_{L^2(K)^d} + \|u\|_{L^2(K)},
\end{align*}
where we have used the definition \ref{NCpoly:defJK} of $\mathcal J_K$ together with Jensen's inequality (to write $\|\overline{u}_K\|_{L^2(K)}\le \|u\|_{L^2(K)}$) in the second line, and the same arguments that led to \eqref{eq:est.nabla.PK} to conclude. The proof of \eqref{NCpoly:eq.bound.JK} is complete. \end{proof}

We can now complete the proof of Theorem \ref{NCpoly:th.approx}.

\begin{proof}[Theorem \ref{NCpoly:th.approx}]
As in the proof of Lemma \ref{NCpoly:lem.bound.interp}, $C$ denotes here a generic constant that can change from one line to the other but depends only on $\rho$. 
Let $K\in\mathcal{M}$ and denote by $q_1$ the $L^2$-orthogonal projection of $u_{|K}$ on $\mathbb{P}^1(K)$. 
By \cite[Theorem 1.45]{hho-book}, we have that
\begin{equation}\label{NCpoly:approx.q1}
\|u-q_1\|_{L^2(K)}+h_K\|\nabla(u-q_1)\|_{L^2(K)^d}\le Ch_K^2 |u|_{H^2(K)}.
\end{equation}
Using the polynomial invariance \eqref{NCpoly:JK.poly} and the triangle inequality, we write, for $s=0,1$,
\[
|u-\mathcal J_K u|_{H^s(K)}=|(u-q_1)-\mathcal J_K (u-q_1)|_{H^s(K)}\le
|u-q_1|_{H^s(K)}+|\mathcal J_K (u-q_1)|_{H^s(K)}.
\]
The boundedness properties \eqref{NCpoly:eq.bound.JK} and \eqref{NCpoly:eq.bound.JK.grad} together with the approximation property \eqref{NCpoly:approx.q1} then yield
\[
|u-\mathcal J_K u|_{H^s(K)}\le C(\|u-q_1\|_{L^2(K)}+h_K^{1-s}\|\nabla(u-q_1)\|_{L^2(K)^d})\le
Ch_K^{2-s}|u|_{H^2(K)}.
\]
Squaring, for each $s=0,1$, this inequality and summing over $K\in\mathcal{M}$ yields the estimate on each term in the left-hand side of \eqref{NCpoly:eq.approx.JD}.
\end{proof}

\subsection{Mass-lumping of the LEPNC method}\label{NCpoly:sec:ml}

As discussed in Section \ref{sec:ml.generic}, approximating non-linear models such as \eqref{eq:stefan.strong} requires the usage of mass-lumping, which necessitates to identify a basis of $V^{\textsc{\tiny LEPNC}}_{\mathfrak{T},0}$ such that the coefficients of $v\in V^{\textsc{\tiny LEPNC}}_{\mathfrak{T},0}$ on this basis represent approximate values of $v$ in some portions of $\Omega$. 

\begin{definition}[Mass-lumping operator for the LEPNC method]\label{def:ml.BPNC}
Let $\varpi\in[0,1]$ be a weight, representing the fraction of mass allocated to the faces. For each $K\in\mathcal M$, create a partition $((K_i)_{i=0,\ldots,d}, (K_\sigma)_{\sigma\in\mathcal F_K})$ of $K$ into $(d+1) + {\rm Card}(\mathcal F_K)$ sets, such that, for all $i=0,\ldots,d$ and $\sigma\in\mathcal F_K$,
\begin{align}
\label{NCpoly:ml.nodes}
\mathbi{s}_i\in\overline{K_i}\,,\quad& \overline{\mathbi{x}}_\sigma\in\overline{K_\sigma},\\
\label{NCpoly:meas.ml.regions}
|K_i|=(1-\varpi)\frac{|K|}{d+1}\,,\quad& 
|K_\sigma|=\varpi\frac{|K|}{{\rm Card}(\mathcal F_K)}.
\end{align}
The mass-lumping operator $\Pi^{\textsc{\tiny LEPNC}}_{\mathfrak{T}}:V^{\textsc{\tiny LEPNC}}_{\mathfrak{T},0}\to L^\infty(\Omega)$ is then defined by: for all $v\in V^{\textsc{\tiny LEPNC}}_{\mathfrak{T},0}$, \[
\Pi^{\textsc{\tiny LEPNC}}_{\mathfrak{T}}v = \sum_{K\in\mathcal{M}}\sum_{i=0}^d v_{K,i}\mathbf{1}_{K_i} + \sum_{\sigma\in\mathcal F_{\rm int}}v_\sigma \mathbf{1}_{K_\sigma},
\]
with $(v_\sigma)_{\sigma\in\mathcal F_{\rm int}}$ and $(v_{K,i})_{K\in\mathcal M,\,i=0,\cdots,d}$  given by \eqref{NCpoly:eq.lambda.sigma}-\eqref{NCpoly:eq.vertex}. 
\end{definition}

\begin{remark}[Shape of the partition of $K$]
Fig. \ref{fig:bpnc.ml} illustrates possible choices of regions $K_i$ and $K_\sigma$. 
In practice, due to the usage of quadrature rules for source terms (see Remark \ref{rem:shape.ml.regions}), the precise shapes of these region are irrelevant. 
Only their measures are required to implement the scheme \eqref{eq:stefan.nc}.
\end{remark}

\begin{figure}
\begin{center}
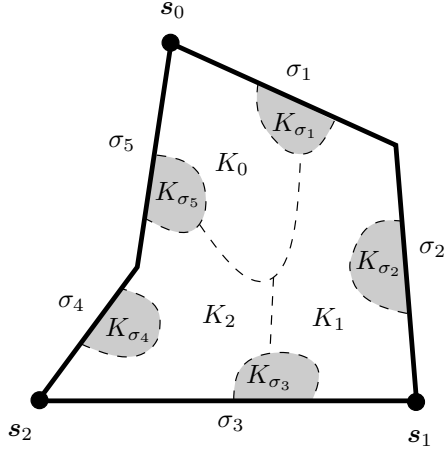
\end{center}
\caption{Regions for mass-lumping of the LEPNC method in dimension $d=2$. Here, $\varpi$ is small and most of the weight has been put on the three chosen vertices $(\mathbi{s}_0,\mathbi{s}_1,\mathbi{s}_2)$.}\label{fig:bpnc.ml}
\end{figure}

The following lemma shows that the above designed mass-lumping technique preserves the approximation properties of the LEPNC, see Lemma \ref{GDM:ml.is.ok}.

\begin{lemma}[Estimate for the mass-lumping operator of the LEPNC]\label{lem:est.ml.BPNC}
Let $\mathfrak{T}$ be a $\rho$-regular polytopal mesh in the sense of Definition \ref{def:rhoregpoly}, and let $\Pi_{\mathfrak{T}}^{\textsc{\tiny LEPNC}}$ be given by Definition \ref{def:ml.BPNC}. Then, there exists $C>0$ depending only on $\rho$ and $d$ such that
\[
\norm{v-\Pi_{\mathfrak{T}}^{\textsc{\tiny LEPNC}}v}{L^2(\Omega)}\le C h_{\mathcal M}\norm{\nabla_{\mathcal M}v}{L^2(\Omega)^d}\quad\forall v\in V^{\textsc{\tiny LEPNC}}_{\mathfrak{T}}.
\]
\end{lemma}

\begin{proof}
In this proof, $C$ is a real number that may vary, but depends only on $\rho$ and $d$.
Let $v\in V^{\textsc{\tiny LEPNC}}_{\mathfrak{T}}$. For all $K\in\mathcal M$, the function $v_{|K}$ is Lipschitz-continuous on $K$ and the $\rho$-regularity of $K$ together with the mean value theorem gives, for all $i=0,\ldots,d$ and $\sigma\in\mathcal F_K$,
\[
|v_{K,i}-v|=|v_{|K}(\mathbi{s}_i)-v|\le Ch_K\norm{\nabla v_{|K}}{L^\infty(K)^d}\mbox{ on $K$}
\]
and
\[
|v_\sigma-v| \le Ch_K\norm{\nabla v_{|K}}{L^\infty(K)^d}\mbox{ on $K$}.
\]
Writing $v_{|K}=\sum_{i=0}^d v\mathbf{1}_{K_i} + \sum_{\sigma\in\mathcal F_{K}}v\mathbf{1}_{K_\sigma}$ and subtracting the definition of $\Pi^{\textsc{\tiny LEPNC}}_{\mathfrak{T}}v$ we infer
\[
|v_{|K}-(\Pi^{\textsc{\tiny LEPNC}}_{\mathfrak{T}}v)_{|K}|\le \sum_{i=0}^d Ch_K\norm{\nabla v_{|K}}{L^\infty(K)^d}\mathbf{1}_{K_i} + \sum_{\sigma\in\mathcal F_{K}}Ch_K\norm{\nabla v_{|K}}{L^\infty(K)^d}\mathbf{1}_{K_\sigma}.
\]
Since $\nabla v_{|K}$ is piecewise polynomial on a regular subdivision of $K$, with a degree bounded above by a positive real number depending only on $\rho$, the inverse Lebesgue inequalities of \cite[Lemma 1.25 and Remark 1.33]{hho-book} yield $\norm{\nabla v_{|K}}{L^\infty(K)^d}\le C |K|^{-\frac12}\norm{\nabla v_{|K}}{L^2(K)^d}$. 
Plugging this estimate into the above relation and using $\sum_{i=0}^d\mathbf{1}_{K_i}+\sum_{\sigma\in\mathcal F_K}\mathbf{1}_{K_\sigma}=1$ on $K$, we infer
\[
|v_{|K}-(\Pi^{\textsc{\tiny LEPNC}}_{\mathfrak{T}}v)_{|K}|\le Ch_K |K|^{-\frac12}\norm{\nabla v_{|K}}{L^2(K)^d}.
\]
The proof is complete by taking the $L^2(K)$-norm of this estimate, squaring, summing over $K\in\mathcal M$ and taking the square root. \end{proof}

\subsection{Convergence results}\label{sec:cvgcelepnc}

Together with the above analysis of the LEPNC properties, the general nonconforming framework of Section \ref{sec:nc.approx.lin} yields the following results. 
We first give an error estimate for the LENPC approximation of the linear problem \eqref{eq:linear}.

\begin{theorem}[Error estimates for the LEPNC approximation]\label{th:error.est.lepnc}
We assume that the solution $\bar u$ of \eqref{ellgenf} and the data $\Lambda$ and $\bm{F}$ in Hypotheses \eqref{hypglin} are such that $\Lambda\nabla\bar u+\bm{F} \in H^1(\Omega)^d$ and $\bar u\in H^2(\Omega)$.
Let $\mathfrak{T}$ be a  $\rho$-regular polytopal mesh in the sense of Definition \ref{def:rhoregpoly}. Let $u$ be the solution of the non-conforming scheme \eqref{eq:nc.lin}, letting $V_{\mathfrak{T},0} = V^{\textsc{\tiny LEPNC}}_{\mathfrak{T},0}$ defined by \eqref{def:V.T}. Then, there exists $C>0$ depending only on $\Omega$, $\rho$ and $\underline{\lambda},\overline{\lambda}$ in \eqref{hyplambda} such that
\begin{equation}\label{eq:errorest.LEPNC}
\norm{\bar u-u}{L^2(\Omega)}+\norm{\nabla \bar u-\nabla_{\mathcal M}u}{L^2(\Omega)^d}
\le C h_{\mathcal M} (\norm{\Lambda\nabla\bar u+\bm{F}}{H^1(\Omega)} +  |u|_{H^2(\Omega)}),
\end{equation}
where $|{\cdot}|_{H^2(\Omega)}$ denotes the $H^2(\Omega)$-seminorm.
\end{theorem}
\begin{proof}
	The result is an immediate consequence of Theorem \ref{th:error.est} and Theorem \ref{NCpoly:th.approx}. 
\end{proof}

Turning to the nonlinear problem \eqref{eq:stefan.weak}, the following theorem states the convergence of the LEPNC method. 

\begin{theorem}[Convergence of the LEPNC method for the Stefan problem]\label{th:cv.stefan.LEPNC}~\\
Let $\rho>0$ be a fixed number, and let $(\mathfrak{T}_m)_{m\in\mathbb{N}}$ be a sequence of $\rho$-regular polytopal mesh polytopal meshes, in the sense of Definition \ref{def:rhoregpoly}, such that $h_{\mathcal M_m}\to 0$ as $m\to\infty$. 

Then, for all $m\in\mathbb{N}$, letting
$V_{\mathfrak{T}_m,0} = V^{\textsc{\tiny LEPNC}}_{\mathfrak{T}_m,0}$  defined by \eqref{def:V.T} and $\Pi_{\mathfrak T_m}=\Pi_{\mathfrak T_m}^{\textsc{\tiny LEPNC}}$ from Definition \ref{def:ml.BPNC}, there exists $u_m$ solution of \eqref{eq:stefan.nc} and, as $m\to\infty$, $\Pi_{\mathfrak{T}_m}^{\textsc{\tiny LEPNC}}\zeta(u_m)\to \zeta(\bar u)$ strongly in $L^2(\Omega)$, $\nabla_{\mathcal M_m}\zeta(u_m)\to\nabla\zeta(\bar u)$ strongly in $L^2(\Omega)^d$, and $\Pi_{\mathfrak{T}_m}^{\textsc{\tiny LEPNC}}u_m\to\bar u$ weakly in $L^2(\Omega)$, where $\bar u$ is a solution to \eqref{eq:stefan.weak}.
\end{theorem}
\begin{proof}
We apply Theorem \ref{th:cv.stefan}. 
Property \eqref{eq:consistency.VT} is a consequence of Theorem \ref{NCpoly:th.approx}, and of the density of $H^2(\Omega)\cap H^1_0(\Omega)$ in $H^1_0(\Omega)$. 
Property \eqref{eq:Pi.comparison} is proven by Lemma \ref{lem:est.ml.BPNC}.
\end{proof}

\subsection{Numerical tests}\label{sec:numer}

We present here some numerical results obtained by the LEPNC method on the linear single-phase incompressible flow \eqref{eq:linear} and on the Stefan/porous medium equation problem \eqref{eq:stefan.strong}, on $\Omega=(0,1)^2$ and with the diffusion tensor $\Lambda={\rm Id}$. 
The schemes we consider are therefore \eqref{eq:nc.lin} and \eqref{eq:stefan.nc} with the space $V^{\textsc{\tiny LEPNC}}_{\mathfrak{T},0}$ and the mass-lumping operator $\Pi^{\textsc{\tiny LEPNC}}_{\mathfrak T}$. 
The tests below were run using the LEPNC implementation available in the HArDCore2D library \cite{HArDCore2D}. 
We note that some of the tests here involve non-homogeneous Dirichlet boundary conditions; adapting the LEPNC scheme to this case is straightforward, and done as for standard non-conforming $\mathbb{P}^1$ finite elements. 
We also refer the interested reader to \cite{CDGGBP20} for a numerical assessment of the LEPNC (and comparison with other methods) on the transient porous medium equation.

Let us first make some remarks relative to the practical implementation of these LEPNC schemes.

\begin{remark}[Choice of implementation unknown for the Stefan model]
Owing to Lemma \ref{NCpoly:lem.basis}, the unknowns for the implementation of the LEPNC represent function values $X_{K,i}$ at the chosen vertices $\mathbi{s}_i$ inside each cell $K\in\mathcal M$, and function values $X_\sigma$ at the center of mass of each face $\sigma\in\mathcal F$ (such values are order 2 approximations of the averages appearing in \eqref{NCpoly:eq.lambda.sigma}). 
When considering the scheme \eqref{eq:stefan.nc} for the Stefan problem and because of the plateaux of $\zeta$, however, these values may not be values of $u$, but sometimes of $\zeta(u)$. 
Specifically, if $\varpi=0$, then the face values of the unknowns $u$ do not appear in the mass-matrix in each Newton iteration on \eqref{eq:stefan.nc}; if we were to use these face values as unknown $X_\sigma$ for the implementation, they would be multiplied in the stiffness matrix by $\zeta'(X_\sigma^{k-1})$, where $X_\sigma^{k-1}$ is the face value at the previous Newton iteration; this factor $\zeta'(X_\sigma^{k-1})$ could vanish, leading to a zero line in the complete linear system. For this reason, when $\varpi=0$, each $X_\sigma$ should represent the value on $\sigma$ of $\zeta(u)$, not $u$; this way, when writing Newton iterations, no linearisation is performed on this unknown in the stiffness matrix, which ensures that it remains invertible. 
For the same reason, if $\varpi=1$, each unknown $X_{K,i}$ should represent values at $\mathbi{s}_i$ of $\zeta(u)$, not $u$. 
We refer the reader to \cite[Remark 3.1]{DE19} for more on this topic.
\end{remark}

\begin{remark}[Static condensation of cell-based degrees of freedom]\label{rem:statcond}
For each $K\in\mathcal{M}$, the basis functions $\{\widetilde{\phi}_{K,i}\,:\,i=0,\ldots,d\}$ have support in $K$. 
In the linear systems to be solved (at each iteration of the Newton algorithm in the case of non-linear problems), the stencil of their associated unknowns therefore only contains the unknowns of the other basis functions related to $K$, and of the basis functions related to the faces of $K$. 
A static condensation process can thus be applied, exactly as in Hybrid High-Order methods (see \cite[Appendix B.3.2]{hho-book}), to eliminate the cell-based unknowns. 
The resulting globally coupled linear system then only involves face-based unknowns, and two faces are in a stencil of this matrix only if they share a cell.
\end{remark}

\begin{remark}[Alternate construction of the basis functions]
Instead of using the no\-dal basis functions $(\psi_{K,i})_{i=0,\ldots,d}$ in \eqref{NCpoly:def.phiKi}, one can instead take the scaled and translated monomial basis functions: $\psi_{K,0}=1$ and $\psi_{K,i}(\mathbi{x})=\frac{x_i-\overline{x}_{K,i}}{h_K}$, where $x_i$ is the $i$-th coordinate of $\mathbi{x}$ and $x_{K,i}$ is the $i$-th coordinate of the centre of mass of $K$. The obtained basis $(\phi_{K,i})_{i=0,\ldots,d}$ can afterwards be transformed by linear combinations into a nodal basis (ensuring that \eqref{NCpoly:dec.v}--\eqref{NCpoly:eq.vertex} holds). 
This implementation is the choice made in the HArDCore library.
\end{remark}

When an analytical solution is available, we present error estimates in the following relative norms:
\[
E_{L^2}:=\frac{\norm{u-\mathcal I_{\mathfrak T}\bar u}{L^2(\Omega)}}{\norm{\mathcal I_{\mathfrak T}\bar u}{L^2(\Omega)}}
\quad\mbox{ and }\quad
E_{H^1}:=\frac{\norm{\nabla_{\mathcal M}(u-{\mathcal I}_{\mathfrak T}\bar u)}{L^2(\Omega)^d}}{\norm{\nabla_{\mathcal M}{\mathcal I}_{\mathfrak T}\bar u}{L^2(\Omega)^d}}
\]
for the linear model, and
\[
E_{L^2,{\rm ml}}:=\frac{\norm{\Pi^{\textsc{\tiny LEPNC}}_{\mathfrak T}(u-\mathcal I_{\mathfrak T}\bar u)}{L^2(\Omega)}}{\norm{\Pi^{\textsc{\tiny LEPNC}}_{\mathfrak T}\mathcal I_{\mathfrak T}\bar u}{L^2(\Omega)}}
\quad\mbox{ and }\quad
E_{H^1,\zeta}:=\frac{\norm{\nabla_{\mathcal M}(\zeta(u)-{\mathcal I}_{\mathfrak T}\zeta(\bar u))}{L^2(\Omega)^d}}{\norm{\nabla_{\mathcal M}{\mathcal I}_{\mathfrak T}\zeta(\bar u)}{L^2(\Omega)^d}}
\]
for the non-linear model; here $\bar u$ is the exact analytical solution to \eqref{eq:stefan.strong}, $u$ is the solution to the LEPNC scheme, ${\mathcal I}_{\mathfrak T}$ is the interpolator defined by \eqref{NCpoly:def:ID}, and $\Pi^{\textsc{\tiny LEPNC}}_{\mathfrak T}$ is the mass-lumping operator given by Definition \ref{def:ml.BPNC}.

The tests have been run using three families of meshes, an example of each is represented in Fig. \ref{fig:ex.meshes}: (mostly) hexagonal meshes, Kershaw meshes and locally refined Cartesian meshes. The last two are taken from the FVCA5 Benchmark \cite{2Dbench}. 
In all the tests we have chosen a mass-lumping weight $\varpi$ of 0 on the edges; tests (not reported here) with other weights show similar results, except that the Newton iterations converge sometimes more slowly when mass is allocated to the edges.

\begin{figure}
\begin{center}
\begin{tabular}{ccc}
\includegraphics[width=0.3\linewidth]{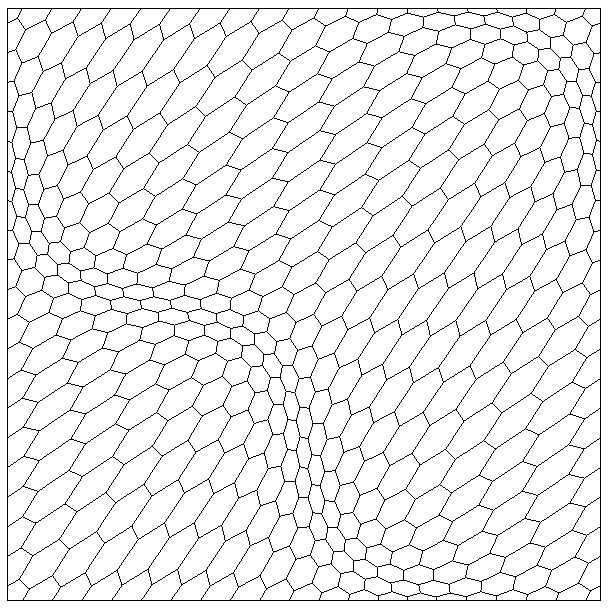}&
\includegraphics[width=0.3\linewidth]{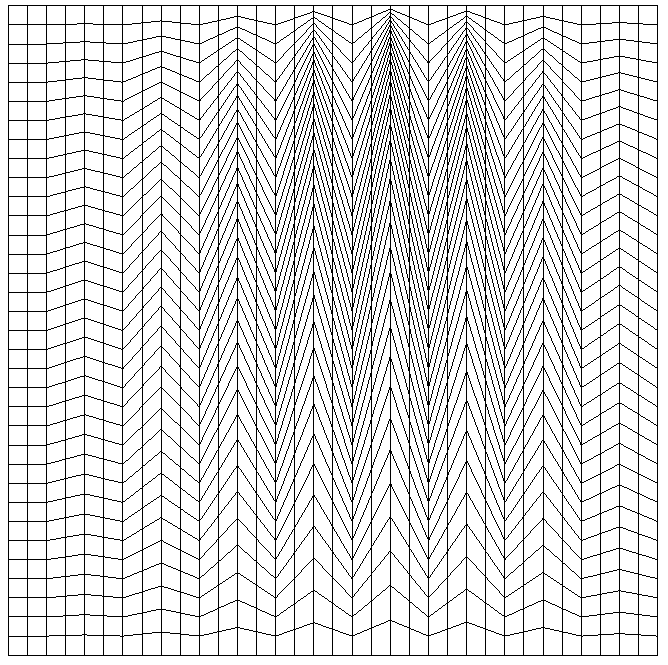}&
\includegraphics[width=0.3\linewidth]{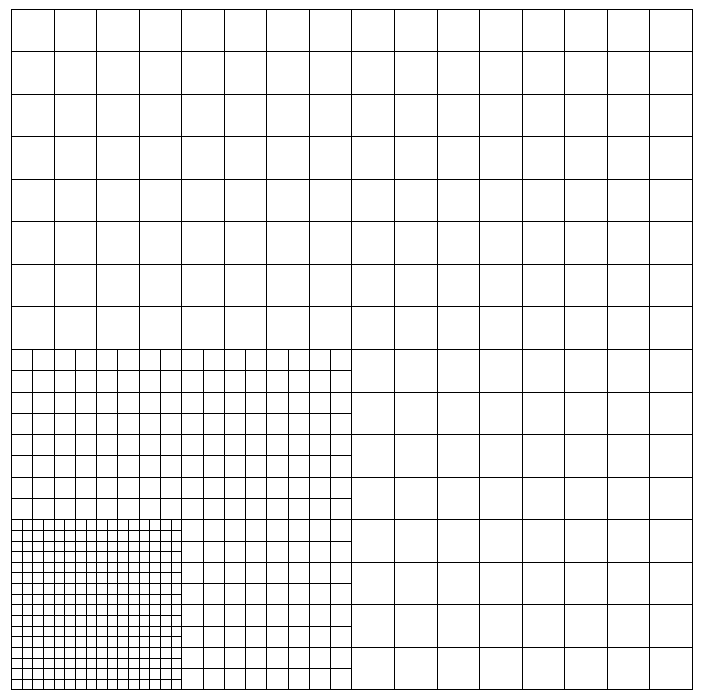}
\end{tabular}
\caption{Examples of members from the mesh families used in numerical tests: hexagonal (left), Kershaw (centre) and locally refined Cartesian (right).}
\label{fig:ex.meshes}
\end{center}
\end{figure}

\subsubsection{Linear single-phase incompressible flow}

We first test the LEPNC method on \eqref{eq:linear} with $\Lambda={\rm Id}$ and exact solution $\bar u(x,y)=\sin(\pi x)\sin(\pi y)$. For comparison, we also present the results obtained with the HHO$(k,\ell)$ method detailed in \cite[Section 5.1]{hho-book}, with degree of edge unknowns $k=0$ and degree of element unknowns $\ell=1$. 
The reason for choosing these particular $(k,\ell)$ is that the HHO$(0,1)$ method has (whether before or after static condensation) the same number of degrees of freedom as the LEPNC method. 
The results for the three families of meshes are presented in Fig. \ref{fig:testL.errors}. Note that for the the HHO$(0,1)$ method, the error $E_{H^1}$ is measured using the discrete $H^1$-norm defined in \cite[Eq. (2.35)]{hho-book}, and $E_{L^2}$ is computed from the $L^2$-norm of the element unknowns.

\begin{figure}\centering
%  \tikzexternaldisable
  \ref{conv.testL}
%  \tikzexternalenable
  \vspace{0.5cm}\\
  \begin{minipage}[b]{0.45\linewidth}
    \centering
    \begin{tikzpicture}[scale=0.65]
      \begin{loglogaxis}[ legend columns=3, legend to name=conv.testL ]
        \addplot [mark=*, blue] table[x=meshsize,y=H1error] {dat/testL/lepnc_kershaw.dat};
        \addplot [mark=*, red] table[x=meshsize,y=H1error] {dat/testL/lepnc_locref.dat};
        \addplot [mark=*, black] table[x=meshsize,y=H1error] {dat/testL/lepnc_hexa.dat};
        \addplot [mark=square*, mark options=solid, blue, dashed] table[x=meshsize,y=H1error] {dat/testL/hho_kershaw.dat};
        \addplot [mark=square*, mark options=solid, red, dashed] table[x=meshsize,y=H1error] {dat/testL/hho_locref.dat};
        \addplot [mark=square*, mark options=solid, black, dashed] table[x=meshsize,y=H1error] {dat/testL/hho_hexa.dat};
        \logLogSlopeTriangle{0.90}{0.4}{0.1}{1}{black};
				\addlegendentry{LEPNC Kershaw}
				\addlegendentry{LEPNC locally refined}
        \addlegendentry{LEPNC hexagonal}
				\addlegendentry{HHO$(0,1)$ Kershaw}
				\addlegendentry{HHO$(0,1)$ locally refined}
        \addlegendentry{HHO$(0,1)$ hexagonal}
      \end{loglogaxis}
    \end{tikzpicture}
    \subcaption{$E_{H^1}$ vs. $h$. \label{fig:testL.errors.H1}}
  \end{minipage}
  \begin{minipage}[b]{0.45\linewidth}
    \centering
    \begin{tikzpicture}[scale=0.65]
      \begin{loglogaxis}[ legend columns=3, legend to name=conv.testL ]
        \addplot [mark=*, blue] table[x=meshsize,y=L2error] {dat/testL/lepnc_kershaw.dat};
        \addplot [mark=*, red] table[x=meshsize,y=L2error] {dat/testL/lepnc_locref.dat};
        \addplot [mark=*, black] table[x=meshsize,y=L2error] {dat/testL/lepnc_hexa.dat};
        \addplot [mark=square*, mark options=solid, blue, dashed] table[x=meshsize,y=L2error] {dat/testL/hho_kershaw.dat};
        \addplot [mark=square*, mark options=solid, red, dashed] table[x=meshsize,y=L2error] {dat/testL/hho_locref.dat};
        \addplot [mark=square*, mark options=solid, black, dashed] table[x=meshsize,y=L2error] {dat/testL/hho_hexa.dat};
        \logLogSlopeTriangle{0.90}{0.4}{0.1}{2}{black};
				\addlegendentry{LEPNC Kershaw}
				\addlegendentry{LEPNC locally refined}
        \addlegendentry{LEPNC hexagonal}
				\addlegendentry{HHO$(0,1)$ Kershaw}
				\addlegendentry{HHO$(0,1)$ locally refined}
        \addlegendentry{HHO$(0,1)$ hexagonal}
      \end{loglogaxis}
    \end{tikzpicture}
    \subcaption{$E_{L^2}$ vs. $h$. \label{fig:testL.errors.L2}}
  \end{minipage}
\caption{Errors versus mesh size for the linear equation. \label{fig:testL.errors}}
\end{figure}

As expected from Theorem \ref{th:error.est.lepnc}, the rate of convergence of the LEPNC scheme in $H^1$-norm is 1 on all three families of meshes. 
An improved rate of order 2 is observed in $L^2$-norm and, even though it is not stated in Theorem \ref{th:error.est.lepnc}, it is also quite expected since LEPNC is close to a lowest-order finite element method (we note that improved $L^2$ estimates can be obtained, using a Nitsche argument, in the context of the GDM \cite{DN16}).

In terms of $H^1$-error, HHO$(0,1)$ seems to over-perform LEPNC on all meshes, especially on distorted ones (Kershaw, hexagonal) where the difference is a full order of magnitude; the difference is less perceptible on more regular meshes like the locally refined ones. 
This is also the case, although much less pronounced (factor 2 instead of a full order of magnitude), in $L^2$-norm on hexagonal and Kershaw meshes; interestingly, the trend is actually reversed on locally refined meshes, with LEPNC providing an $L^2$-error about five times smaller than HHO$(0,1)$, indicating that LEPNC seems to produce a better approximation of the solution itself (if not its gradient) on regular meshes. 
Of course, all these comparisons must be taken with a grain of salt since they do not exactly use the same norms. 
Additionally, it should be noted that the HHO$(0,1)$ scheme does not readily produce an explicit function that embeds all the methods' design (it is, in this sense, more of a \emph{virtual} method), whereas LEPNC does.

\subsubsection{Stefan problem}

We consider the problem \eqref{eq:stefan.weak} with the following Stefan non-linearity:
\[
\zeta(s)=\left\{\begin{array}{ll} s &\mbox{ if $s\le 0$},\\
0&\mbox{ if $0\le s\le 1$},\\
s-1&\mbox{ if $s\ge 1$}.
\end{array}\right.
\]

\textsc{Test S1.} For this test, we take an exact smooth solution $\bar u$ such that $\zeta(\bar u)$ is also smooth, but not trivial (the solution $\bar u$ crosses the value $0$ at which $\zeta$ is not differentiable). 
Setting $s(x,y)=\frac{x+y}{\sqrt{2}}$ the coordinate along the first diagonal, the exact solution is $\bar u(x,y)=(s(x,y)-0.5)^3$. The functions $\bar u$ and $\zeta(\bar u)$ are represented in Fig. \ref{fig:testS1.sol}

\begin{figure}
\begin{center}
\begin{tabular}{cc}
\includegraphics[width=0.4\linewidth]{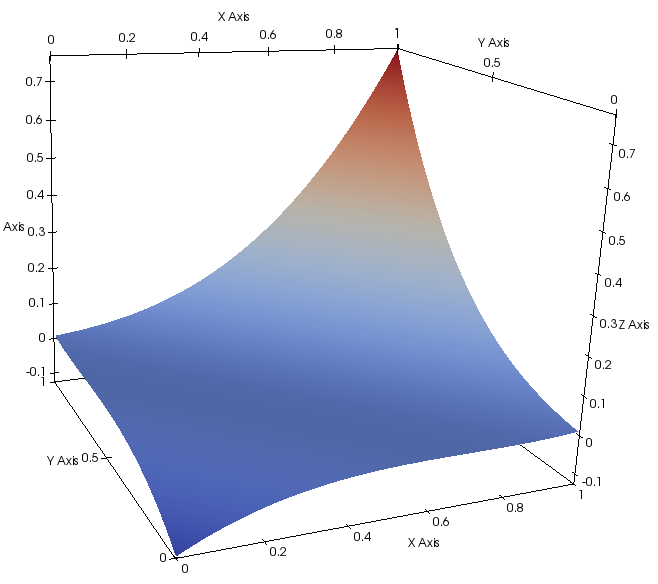}&
\includegraphics[width=0.4\linewidth]{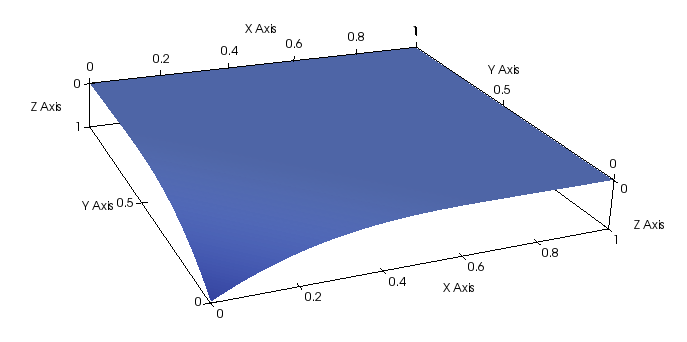}
\end{tabular}
\caption{Exact solution $\bar u$ (left) and $\zeta(\bar u)$ (right) for Test S1.}
\label{fig:testS1.sol}
\end{center}
\end{figure}

The convergence graphs are given in Fig. \ref{fig:testS1.errors}. 
For solutions that are piecewise smooth on the mesh, the analysis of \cite{DE19} shows that, for a low-order scheme as the LEPNC, the expected rate of convergence in energy error $E_{H^1,\zeta}$ for the regular variable $\zeta(u)$ is $\mathcal O(h)$, which corresponds to the rate observed for all three families in Fig. \ref{fig:testS1.errors.H1}. 
The convergence rate in mass-lumped $L^2$-norm on the $u$ variable is always larger than one: it is almost $2$ for the hexagonal and locally refined mesh families, and around 1.5 for the Kershaw family.
 This convergence is however less regular than the convergence on the variable $\zeta(u)$.

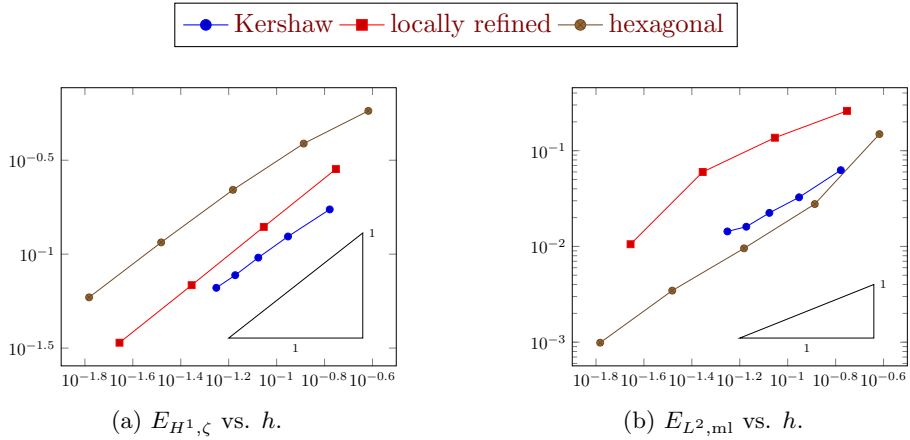
\begin{figure}\centering
%  \tikzexternaldisable
  \ref{conv.testS1}
%  \tikzexternalenable
  \vspace{0.5cm}\\
  \begin{minipage}[b]{0.45\linewidth}
    \centering
    \begin{tikzpicture}[scale=0.65]
      \begin{loglogaxis}[ legend columns=-1, legend to name=conv.testS1 ]
        \addplot table[x=meshsize,y=EnergyError] {dat/testS1/kershaw.dat};
        \addplot table[x=meshsize,y=EnergyError] {dat/testS1/locref.dat};
        \addplot table[x=meshsize,y=EnergyError] {dat/testS1/hexa.dat};
        \logLogSlopeTriangle{0.90}{0.4}{0.1}{1}{black};
				\addlegendentry{Kershaw}
				\addlegendentry{locally refined}
        \addlegendentry{hexagonal}
      \end{loglogaxis}
    \end{tikzpicture}
    \subcaption{$E_{H^1,\zeta}$ vs. $h$. \label{fig:testS1.errors.H1}}
  \end{minipage}
  \begin{minipage}[b]{0.45\linewidth}
    \centering
    \begin{tikzpicture}[scale=0.65]
      \begin{loglogaxis}[ legend columns=-1, legend to name=conv.testS1 ]
        \addplot table[x=meshsize,y=L2error_ml] {dat/testS1/kershaw.dat};
        \addplot table[x=meshsize,y=L2error_ml] {dat/testS1/locref.dat};
        \addplot table[x=meshsize,y=L2error_ml] {dat/testS1/hexa.dat};
        \logLogSlopeTriangle{0.90}{0.4}{0.1}{1}{black};
				\addlegendentry{Kershaw}
				\addlegendentry{locally refined}
        \addlegendentry{hexagonal}
      \end{loglogaxis}
    \end{tikzpicture}
    \subcaption{$E_{L^2,{\rm ml}}$ vs. $h$. \label{fig:testS1.errors.L2}}
  \end{minipage}
\caption{Errors versus mesh size for Test S1. \label{fig:testS1.errors}}
\end{figure}

\textsc{Test S2.} The previous test is not representative of the typical behaviour of solutions to Stefan problems. 
In the general case, and in particular with null source terms, these solutions $\bar u$ are discontinuous in the range of values where $\zeta$ remains constant, which therefore does not prevent $\zeta(\bar u)$ from being continuous. 
This next test case, taken from \cite{DE19}, displays such a behaviour. 
Setting $\gamma=\frac13$, the exact solution is
\[
\bar u(x,y)=\cosh(s(x,y)-\gamma)\mbox{ if $s(x,y)\ge \gamma$}\,,\quad \bar u(x,y)=0\mbox{ if $s(x,y)<\gamma$},
\]
where, as in Test S1, $s(x,y)=\frac{x+y}{\sqrt{2}}$ is the coordinate along the first diagonal.
This solution is discontinuous along the line $s(x,y)=\gamma$, but $\zeta(\bar u)$ is continuous (and even in $H^2(\Omega)$); see Fig. \ref{fig:testS2.sol}. 
This function corresponds to a zero source term in \eqref{eq:stefan.strong}.

\begin{figure}
\begin{center}
\begin{tabular}{cc}
\includegraphics[width=0.4\linewidth]{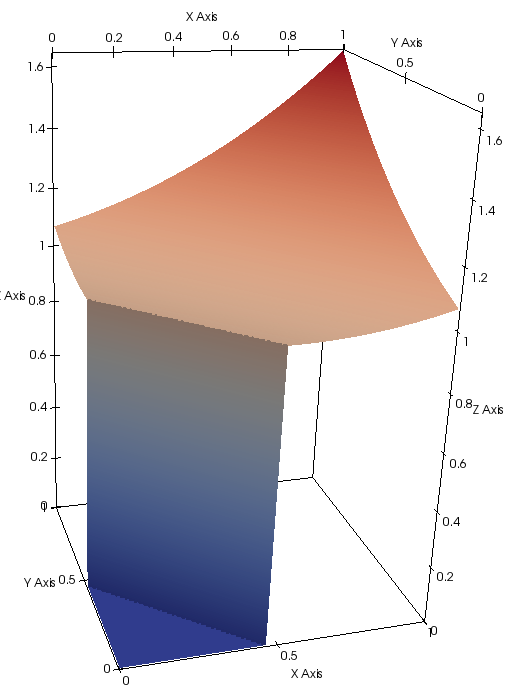}&
\includegraphics[width=0.4\linewidth]{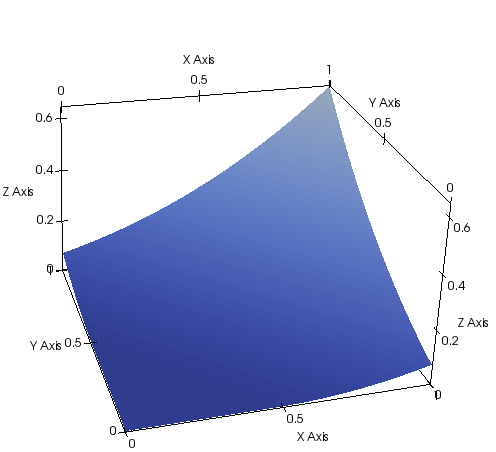}
\end{tabular}
\caption{Exact solution $\bar u$ (left) and $\zeta(\bar u)$ (right) for Test S2.}
\label{fig:testS2.sol}
\end{center}
\end{figure}

The convergence results are presented in Fig. \ref{fig:testS2.errors}. 
As expected from the results of \cite{DE19}, we observe in Fig. \ref{fig:testS2.errors.H1} an estimate of the kind $E_{H^1,\zeta}=\mathcal O(h)$. 
The convergence rate in mass-lumped $L^2$ error $E_{L^2,{\rm ml}}$ for the variable $u$ is however much lower (and, as in Test S1, rather irregular), which is expected since $u$ is discontinuous; the overall convergence rate of $E_{L^2,{\rm ml}}$ is about $\mathcal O(h^{0.6})$ for all mesh families. 
Fig. \ref{fig:testS2.approx_sol} shows the approximate variables $u$ and $\zeta(u)$ obtained on the second hexagonal mesh in the family; the discontinuity of $\bar u$, typical in Stefan's problems, clearly impacts the convergence on this variable.

\begin{figure}\centering
%  \tikzexternaldisable
  \ref{conv.testS2}
%  \tikzexternalenable
  \vspace{0.5cm}\\
  \begin{minipage}[b]{0.45\linewidth}
    \centering
    \begin{tikzpicture}[scale=0.65]
      \begin{loglogaxis}[ legend columns=-1, legend to name=conv.testS2 ]
        \addplot table[x=meshsize,y=EnergyError] {dat/testS2/kershaw.dat};
        \addplot table[x=meshsize,y=EnergyError] {dat/testS2/locref.dat};
        \addplot table[x=meshsize,y=EnergyError] {dat/testS2/hexa.dat};
        \logLogSlopeTriangle{0.90}{0.4}{0.1}{1}{black};
				\addlegendentry{Kershaw}
				\addlegendentry{locally refined}
        \addlegendentry{hexagonal}
      \end{loglogaxis}
    \end{tikzpicture}
    \subcaption{$E_{H^1,\zeta}$ vs. $h$. \label{fig:testS2.errors.H1}}
  \end{minipage}
  \begin{minipage}[b]{0.45\linewidth}
    \centering
    \begin{tikzpicture}[scale=0.65]
      \begin{loglogaxis}[ legend columns=-1, legend to name=conv.testS2 ]
        \addplot table[x=meshsize,y=L2error_ml] {dat/testS2/kershaw.dat};
        \addplot table[x=meshsize,y=L2error_ml] {dat/testS2/locref.dat};
        \addplot table[x=meshsize,y=L2error_ml] {dat/testS2/hexa.dat};
        \logLogSlopeTriangle{0.90}{0.4}{0.1}{1}{black};
				\addlegendentry{Kershaw}
				\addlegendentry{locally refined}
        \addlegendentry{hexagonal}
      \end{loglogaxis}
    \end{tikzpicture}
    \subcaption{$E_{L^2,{\rm ml}}$ vs. $h$. \label{fig:testS2.errors.L2}}
  \end{minipage}
\caption{Errors versus mesh size for Test S2. \label{fig:testS2.errors}}
\end{figure}
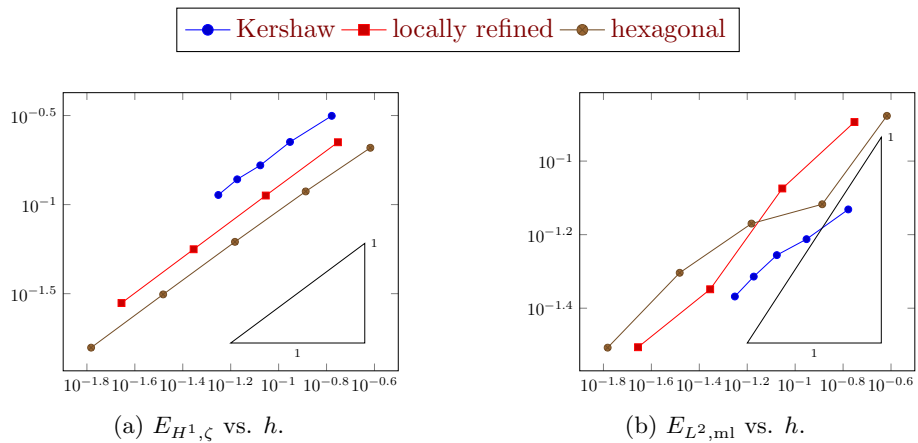

\begin{figure}
\begin{center}
\begin{tabular}{cc}
\includegraphics[width=0.4\linewidth]{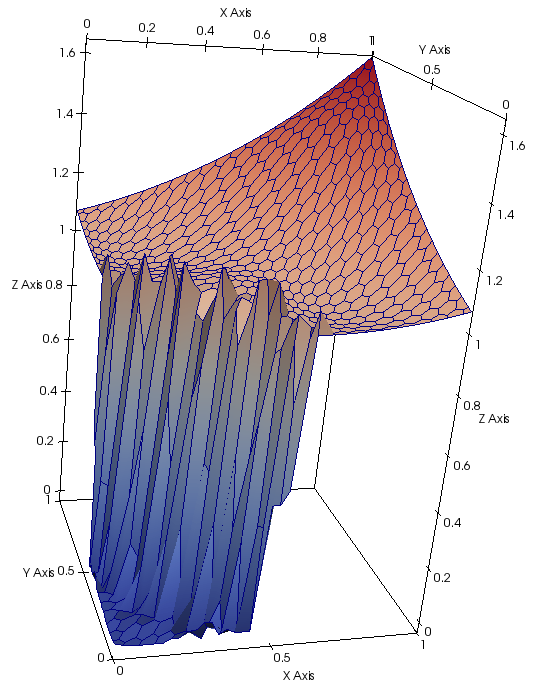}&
\includegraphics[width=0.4\linewidth]{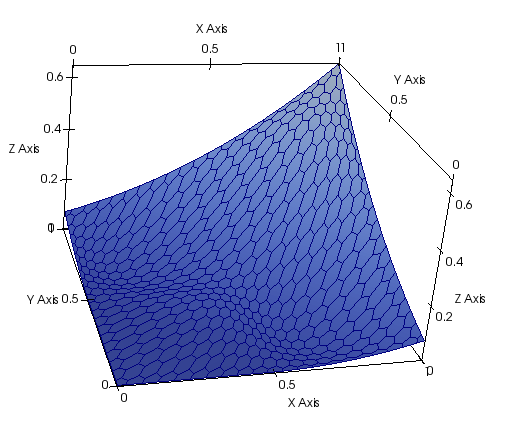}
\end{tabular}
\caption{Approximate solution $u$ (left) and $\zeta(u)$ (right) obtained on the second member of the hexagonal mesh family in Test S2.}
\label{fig:testS2.approx_sol}
\end{center}
\end{figure}

\subsubsection{Porous medium equation}

We now consider the stationary porous medium equation, corresponding to \eqref{eq:stefan.strong} with non-linearity
\[
\zeta(s)=|s|^{m-1}s\mbox{ with $m\ge 1$}.
\]
\textsc{Test P1}.
For this test, the exact solutions $\bar u$ and $\zeta(\bar u)$ are both smooth. We take $\bar u(x,y)=\sin(\pi x)\sin(\pi y)$, and $m\in\{1,2,3,4\}$. Note that the case $m=1$ actually corresponds to $\zeta(s)=s$, so \eqref{eq:stefan.strong} is the linear equation \eqref{eq:linear} with an added reaction term $u$. 
The results of the test, on the same Kershaw, locally refined and hexagonal meshes as in Tests S1 and S2, are presented in Fig. \ref{fig:testP1.errors}. 

Looking first at the case $m=1$, we notice that the results are worse on the Kershaw meshes; despite the smoothness of the solution, the distortion of these meshes impact the approximation error negatively. 
We still see an order $\mathcal O(h)$ convergence in both energy and mass-lumped $L^2$ norm; this is expected for the energy error given that LEPNC is a low-order scheme, but one could have hoped to see a super-convergence effect in the $L^2$-norm.
On the contrary, for locally refined and hexagonal meshes, this super-convergence is visible and the $L^2$-norm error decays as $\mathcal O(h^2)$, while the energy norm decays as $\mathcal O(h)$.

Considering now the nonlinear cases $m=2,3,4$, we see that the energy error still decays as $h$ for the locally refined and hexagonal meshes. However, the $L^2$-norm error no longer super-converges with an order 2, but rather with an order 1.5. The results for the Kershaw meshes show much lower convergence rates. For $m=2$ rate for the $L^2$-norm error is still close to 1, but the energy error only decays as about $\mathcal O(h^{0.5})$. 
For $m=3,4$, the rates in $L^2$-norm and energy error are respectively 0.5 and 0.3 -- at least at the considered mesh sizes. 
Looking at the pictures it seems that the rate in energy norm has a tendency to increase towards the last meshes in the Kershaw family. 
It should be mentioned here that for certain cases (typically, the finest hexagonal or Kershaw meshes, with $m=3,4$), a straightforward Newton algorithm does not converge and relaxation has to be applied.

\begin{figure}\centering
%  \tikzexternaldisable
  \ref{conv.testP1}
%  \tikzexternalenable
  \vspace{0.5cm}\\
  \begin{minipage}[b]{0.45\linewidth}
    \centering
    \begin{tikzpicture}[scale=0.65]
      \begin{loglogaxis}[ legend columns=3, legend to name=conv.testP1 ]
        \addplot [mark=*, blue] table[x=meshsize,y=EnergyError] {dat/testP1/m1_kershaw.dat};
        \addplot [mark=square*, black] table[x=meshsize,y=EnergyError] {dat/testP1/m1_locref.dat};
        \addplot [mark=*, red] table[x=meshsize,y=EnergyError] {dat/testP1/m1_hexa.dat};
        \addplot [mark=*, mark options=solid, blue, dashed] table[x=meshsize,y=L2error_ml] {dat/testP1/m1_kershaw.dat};
        \addplot [mark=square*, mark options=solid, black, dashed] table[x=meshsize,y=L2error_ml] {dat/testP1/m1_locref.dat};
        \addplot [mark=*, mark options=solid, red, dashed] table[x=meshsize,y=L2error_ml] {dat/testP1/m1_hexa.dat};
        \logLogSlopeTriangle{0.90}{0.4}{0.1}{1}{black};
        \logLogSlopeTriangle{0.90}{0.4}{0.1}{2}{black};
				\addlegendentry{$E_{H^1,\zeta}$, Kershaw}
				\addlegendentry{$E_{H^1,\zeta}$, locally refined}
        \addlegendentry{$E_{H^1,\zeta}$, hexagonal}
				\addlegendentry{$E_{L^2, {\rm ml}}$, Kershaw}
				\addlegendentry{$E_{L^2, {\rm ml}}$, locally refined}
        \addlegendentry{$E_{L^2, {\rm ml}}$, hexagonal}
      \end{loglogaxis}
    \end{tikzpicture}
    \subcaption{$m=1$ (linear equation). \label{fig:testP1.errors.m1}}
  \end{minipage}
  \begin{minipage}[b]{0.45\linewidth}
    \centering
    \begin{tikzpicture}[scale=0.65]
      \begin{loglogaxis}[ legend columns=3, legend to name=conv.testP1 ]
        \addplot [mark=*, blue] table[x=meshsize,y=EnergyError] {dat/testP1/m2_kershaw.dat};
        \addplot [mark=square*, black] table[x=meshsize,y=EnergyError] {dat/testP1/m2_locref.dat};
        \addplot [mark=*, red] table[x=meshsize,y=EnergyError] {dat/testP1/m2_hexa.dat};
        \addplot [mark=*, mark options=solid, blue, dashed] table[x=meshsize,y=L2error_ml] {dat/testP1/m2_kershaw.dat};
        \addplot [mark=square*, mark options=solid, black, dashed] table[x=meshsize,y=L2error_ml] {dat/testP1/m2_locref.dat};
        \addplot [mark=*, mark options=solid, red, dashed] table[x=meshsize,y=L2error_ml] {dat/testP1/m2_hexa.dat};
        \logLogSlopeTriangle{0.90}{0.4}{0.1}{0.5}{black};
        \logLogSlopeTriangle{0.90}{0.4}{0.1}{1}{black};
        \logLogSlopeTriangle{0.90}{0.4}{0.1}{1.5}{black};
				\addlegendentry{$E_{H^1,\zeta}$, Kershaw}
				\addlegendentry{$E_{H^1,\zeta}$, locally refined}
        \addlegendentry{$E_{H^1,\zeta}$, hexagonal}
				\addlegendentry{$E_{L^2, {\rm ml}}$, Kershaw}
				\addlegendentry{$E_{L^2, {\rm ml}}$, locally refined}
        \addlegendentry{$E_{L^2, {\rm ml}}$, hexagonal}
      \end{loglogaxis}
    \end{tikzpicture}
    \subcaption{$m=2$. \label{fig:testP1.errors.m2}}
  \end{minipage}
  \\
  \begin{minipage}[b]{0.45\linewidth}
    \centering
    \begin{tikzpicture}[scale=0.65]
      \begin{loglogaxis}[ legend columns=3, legend to name=conv.testP1 ]
        \addplot [mark=*, blue] table[x=meshsize,y=EnergyError] {dat/testP1/m3_kershaw.dat};
        \addplot [mark=square*, black] table[x=meshsize,y=EnergyError] {dat/testP1/m3_locref.dat};
        \addplot [mark=*, red] table[x=meshsize,y=EnergyError] {dat/testP1/m3_hexa.dat};
        \addplot [mark=*, mark options=solid, blue, dashed] table[x=meshsize,y=L2error_ml] {dat/testP1/m3_kershaw.dat};
        \addplot [mark=square*, mark options=solid, black, dashed] table[x=meshsize,y=L2error_ml] {dat/testP1/m3_locref.dat};
        \addplot [mark=*, mark options=solid, red, dashed] table[x=meshsize,y=L2error_ml] {dat/testP1/m3_hexa.dat};
        \logLogSlopeTriangle{0.90}{0.4}{0.1}{0.5}{black};
        \logLogSlopeTriangle{0.90}{0.4}{0.1}{1}{black};
        \logLogSlopeTriangle{0.90}{0.4}{0.1}{1.5}{black};
				\addlegendentry{$E_{H^1,\zeta}$, Kershaw}
				\addlegendentry{$E_{H^1,\zeta}$, locally refined}
        \addlegendentry{$E_{H^1,\zeta}$, hexagonal}
				\addlegendentry{$E_{L^2, {\rm ml}}$, Kershaw}
				\addlegendentry{$E_{L^2, {\rm ml}}$, locally refined}
        \addlegendentry{$E_{L^2, {\rm ml}}$, hexagonal}
      \end{loglogaxis}
    \end{tikzpicture}
    \subcaption{$m=3$. \label{fig:testP1.errors.m3}}
  \end{minipage}
  \begin{minipage}[b]{0.45\linewidth}
    \centering
    \begin{tikzpicture}[scale=0.65]
      \begin{loglogaxis}[ legend columns=3, legend to name=conv.testP1 ]
        \addplot [mark=*, blue] table[x=meshsize,y=EnergyError] {dat/testP1/m4_kershaw.dat};
        \addplot [mark=square*, black] table[x=meshsize,y=EnergyError] {dat/testP1/m4_locref.dat};
        \addplot [mark=*, red] table[x=meshsize,y=EnergyError] {dat/testP1/m4_hexa.dat};
        \addplot [mark=*, mark options=solid, blue, dashed] table[x=meshsize,y=L2error_ml] {dat/testP1/m4_kershaw.dat};
        \addplot [mark=square*, mark options=solid, black, dashed] table[x=meshsize,y=L2error_ml] {dat/testP1/m4_locref.dat};
        \addplot [mark=*, mark options=solid, red, dashed] table[x=meshsize,y=L2error_ml] {dat/testP1/m4_hexa.dat};
        \logLogSlopeTriangle{0.90}{0.4}{0.1}{0.5}{black};
        \logLogSlopeTriangle{0.90}{0.4}{0.1}{1}{black};
        \logLogSlopeTriangle{0.90}{0.4}{0.1}{1.5}{black};
				\addlegendentry{$E_{H^1,\zeta}$, Kershaw}
				\addlegendentry{$E_{H^1,\zeta}$, locally refined}
        \addlegendentry{$E_{H^1,\zeta}$, hexagonal}
				\addlegendentry{$E_{L^2, {\rm ml}}$, Kershaw}
				\addlegendentry{$E_{L^2, {\rm ml}}$, locally refined}
        \addlegendentry{$E_{L^2, {\rm ml}}$, hexagonal}
      \end{loglogaxis}
    \end{tikzpicture}
    \subcaption{$m=4$. \label{fig:testP1.errors.m4}}
  \end{minipage}
\caption{Errors versus mesh size for Test P1. \label{fig:testP1.errors}}
\end{figure}
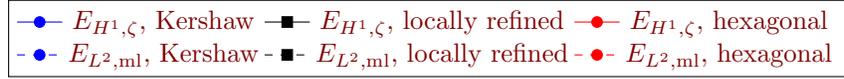
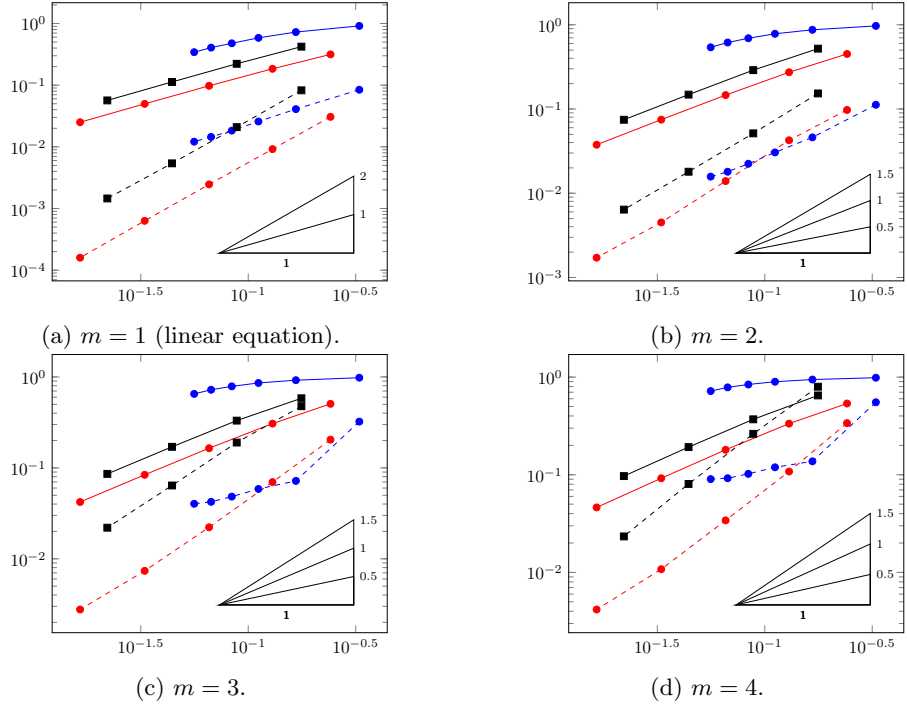

\textsc{Test P2}. 
This test features a less regular exact solution $\bar u$. 
We take $\bar u(x,y)=\max(\rho^2-r(x,y)^2,0)$, where $\rho=0.3$ and $r(x,y)^2=(x-0.5)^2+(y-0.5)^2$.  
In the domain $\Omega$, the graph of $\bar u$ is the tip of a paraboloid; this solution belongs to $H^1(\Omega)$ but not to $H^2(\Omega)$. 
We take $m=2$, so $\zeta(\bar u)\in H^2(\Omega)$. 
For this value of $m$, the singularity of $\bar u$ at the circle $r(x,y)^2=\rho^2$ is typical of the singularity exhibited by the Barenblatt solution in the transient setting \cite{Barenblatt52,Vazquez2007}. 
The results are presented in Fig. \ref{fig:testP2.errors}. 
As in Test P1, we see that the energy error decays as $\mathcal O(h)$, except for the very distorted Kershaw meshes for which a rate of about 0.3 is achieved with the last two meshes (further refinement might improve that rate). 
In terms of the $L^2$-error, all three mesh families lead to a rate of convergence of about 1. 
Even for the relatively regular mesh families (hexahedral, locally refined), no super-convergence is observed. 
This is somehow expected given that the exact solution is not $H^2$-regular.

\begin{figure}\centering
%  \tikzexternaldisable
  \ref{conv.testP2}
%  \tikzexternalenable
  \vspace{0.5cm}\\
  \begin{minipage}[b]{0.45\linewidth}
    \centering
    \begin{tikzpicture}[scale=0.65]
      \begin{loglogaxis}[ legend columns=-1, legend to name=conv.testP2 ]
        \addplot table[x=meshsize,y=EnergyError] {dat/testP2/kershaw.dat};
        \addplot table[x=meshsize,y=EnergyError] {dat/testP2/locref.dat};
        \addplot table[x=meshsize,y=EnergyError] {dat/testP2/hexa.dat};
        \logLogSlopeTriangle{0.90}{0.4}{0.1}{1}{black};
				\addlegendentry{Kershaw}
				\addlegendentry{locally refined}
        \addlegendentry{hexagonal}
      \end{loglogaxis}
    \end{tikzpicture}
    \subcaption{$E_{H^1,\zeta}$ vs. $h$. \label{fig:testP2.errors.H1}}
  \end{minipage}
  \begin{minipage}[b]{0.45\linewidth}
    \centering
    \begin{tikzpicture}[scale=0.65]
      \begin{loglogaxis}[ legend columns=-1, legend to name=conv.testP2 ]
        \addplot table[x=meshsize,y=L2error_ml] {dat/testP2/kershaw.dat};
        \addplot table[x=meshsize,y=L2error_ml] {dat/testP2/locref.dat};
        \addplot table[x=meshsize,y=L2error_ml] {dat/testP2/hexa.dat};
        \logLogSlopeTriangle{0.90}{0.4}{0.1}{1}{black};
				\addlegendentry{Kershaw}
				\addlegendentry{locally refined}
        \addlegendentry{hexagonal}
      \end{loglogaxis}
    \end{tikzpicture}
    \subcaption{$E_{L^2,{\rm ml}}$ vs. $h$. \label{fig:testP2.errors.L2}}
  \end{minipage}
\caption{Errors versus mesh size for Test P2. \label{fig:testP2.errors}}
\end{figure}
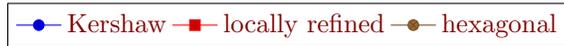
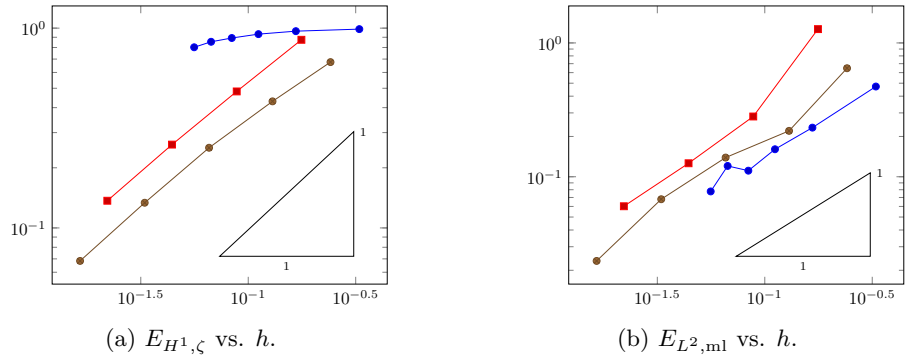

\section{Analysis of polytopal non-conforming finite element schemes}\label{sec:prop.NC}

Polytopal non-conforming finite element schemes are gradient discretisation methods (GDM) and, as such, enjoy all the error estimates and convergence results of GDMs. 
We recall here the notion of GDM and associated results, which yield in particular the theorems \ref{th:error.est} and \ref{th:cv.stefan}. 
Most of the following material is taken from \cite[Section 9.1]{gdm}.

\subsection{Gradient discretisation method}\label{sec:gdm}

The GDM is a generic framework for designing and analysing numerical schemes for elliptic and parabolic problems (although extensions to linear advection is also possible \cite{DEGH19}). 
It consists in replacing, in the weak formulation of the model, the continuous space and operator by their discrete analogues given by a gradient discretisation (GD).

\begin{definition}[Gradient discretisation for homogeneous Dirichlet BCs]
A gradient discretisation for homogeneous Dirichlet boundary conditions is a triplet $\mathcal D=(X_{\mathcal D,0},\Pi_{\mathcal D},\nabla_{\mathcal D})$ where
\begin{itemize}
\item $X_{\mathcal D,0}$ is a finite-dimensional space of unknowns, that encodes the homogeneous boundary conditions,
\item $\Pi_{\mathcal D}:X_{\mathcal D,0}\to L^2(\Omega)$ is a linear operator that reconstructs a function from a vector of unknowns,
\item $\nabla_{\mathcal D}:X_{\mathcal D,0}\to L^2(\Omega)^d$ is a linear operator that reconstructs a ``gradient'' from a vector of unknowns; it must be chosen such that $\norm{\nabla_{\mathcal D} \cdot}{L^2(\Omega)^d}$ is a norm on $X_{\mathcal D,0}$.
\end{itemize}
A gradient discretisation $\mathcal D$ is said to have a piecewise constant reconstruction if there exists a basis $(\mathbf{e}_i)_{i\in I}$ of $X_{\mathcal D,0}$ and disjoint subsets $(U_i)_{i\in I}$ of $\Omega$ such that
\begin{equation}\label{eq:Pi.pwcst}
\Pi_{\mathcal D}v=\sum_{i\in I}v_i\mathbf{1}_{U_i}\quad\forall v=\sum_{i\in I}v_i\mathbf{e}_i\in X_{\mathcal D,0},
\end{equation}
where $\mathbf{1}_{U_i}$ is the characteristic function of $U_i$ (equal to $1$ in this set and to $0$ elsewhere).
\end{definition}

Once a GD $\mathcal D$ is chosen, a gradient scheme (GS) for the linear diffusion problem \eqref{ellgenf} is obtained by writing:
\begin{equation}\begin{array}{l}
\mbox{Find $u \in X_{\mathcal D,0}$ such that, }\forall v \in X_{\mathcal D,0},\\
\displaystyle\int_\Omega \Lambda\nabla_{\mathcal D} u\cdot\nabla_{\mathcal D} v d\mathbi{x}  
= \int_\Omega f\Pi_{\mathcal D}v d\mathbi{x} 
-\int_\Omega \bm{F}\cdot\nabla_{\mathcal D} v d\mathbi{x}.
\end{array}\label{eq:gs.lin}\end{equation}
If $\mathcal D$ has a piecewise constant reconstruction, then it makes sense, for a generic function $g:\mathbb{R}\to\mathbb{R}$ and $v\in X_{\mathcal D,0}$, to define $g(v)\in X_{\mathcal D,0}$ component-by-component: if $v=\sum_{i\in I}v_i\mathbf{e}_i$, then $g(v)=\sum_{i\in I}g(v_i)\mathbf{e}_i$. This definition is justified by the following commutation property, coming from \eqref{eq:Pi.pwcst}:
\[
\Pi_{\mathcal D}g(v)=g(\Pi_{\mathcal D}v)\quad\forall v\in X_{\mathcal D,0}.
\]
Then, a GS for the non-linear model \eqref{eq:stefan.weak} is obtained writing
\begin{equation}\begin{array}{l}
\mbox{Find $u \in X_{\mathcal D,0}$ such that, }\forall v \in X_{\mathcal D,0},\\
\displaystyle\int_\Omega\left(\Pi_{\mathcal D} u\,\Pi_{\mathcal D} v + \Lambda\nabla_{\mathcal D}\zeta(u)\cdot\nabla_{\mathcal D} v\right) d\mathbi{x}  
= \int_\Omega f\Pi_{\mathcal D}v d\mathbi{x} 
-\int_\Omega \bm{F}\cdot\nabla_{\mathcal D} v d\mathbi{x}.
\end{array}\label{eq:stefan.gs}\end{equation}

The accuracy and convergence of a GS is assessed through the following quantities and notions.
\begin{enumerate}
\item\emph{Coercivity}. 
The discrete Poincar\'e constant of a GD $\mathcal D$ is
\[
C_{\mathcal D}:=\max_{v\in X_{\mathcal  D,0}}\frac{\norm{\Pi_{\mathcal D}v}{L^2(\Omega)}}{\norm{\nabla_{\mathcal D}v}{L^2(\Omega)^d}}.
\]
A sequence $(\mathcal D_m)_{m\in\mathbb{N}}$ is \emph{coercive} if $(C_{\mathcal D_m})_{m\in\mathbb{N}}$ is bounded.
\item \emph{Consistency}. The interpolation error of a GD $\mathcal D$ is 
\[
S_{\mathcal D}(\phi):=\min_{v\in X_{\mathcal D,0}}\left(\norm{\Pi_{\mathcal D}v-\phi}{L^2(\Omega)}+
\norm{\nabla_{\mathcal D}v-\nabla\phi}{L^2(\Omega)^d}\right)\quad\forall\phi\in H^1_0(\Omega).
\]
A sequence $(\mathcal D_m)_{m\in\mathbb{N}}$ is \emph{consistent} if $S_{\mathcal D_m}(\phi)\to 0$ as $m\to\infty$, for all $\phi\in H^1_0(\Omega)$.
\item \emph{Limit-conformity}. The defect of conformity of a GD $\mathcal D$ is 
\[
W_{\mathcal D}(\bm{\psi}):=\max_{v\in X_{\mathcal D,0}\backslash\{0\}}\frac{1}{\norm{\nabla_{\mathcal D}v}{L^2(\Omega)^d}}\left|\int_\Omega \Pi_{\mathcal D}v\mathop{\rm div}\bm{\psi}+\nabla_{\mathcal D}v\cdot\bm{\psi}\mathbi{x}\right|
\quad\forall \bm{\psi}\in H_{\mathop{\rm div}}(\Omega).
\]
A sequence $(\mathcal D_m)_{m\in\mathbb{N}}$ is \emph{limit-conforming} if $W_{\mathcal D_m}(\bm{\psi})\to 0$ as $m\to\infty$, for all $\bm{\psi}\in H_{\mathop{\rm div}}(\Omega)$.
\item \emph{Compactness}. A sequence $(\mathcal D_m)_{m\in\mathbb{N}}$ is \emph{compact} if, for any $(v_m)_{m\in\mathbb{N}}$ such that $v_m\in X_{\mathcal D_m,0}$ for all $m\in\mathbb{N}$ and $(\norm{\nabla_{\mathcal D_m}v}{L^2(\Omega)^d})_{m\in\mathbb{N}}$ is bounded, the sequence $(\Pi_{\mathcal D_m}v)_{m\in\mathbb{N}}$ is relatively compact in $L^2(\Omega)$.
\end{enumerate}

We then recall an error estimate for the linear model and a convergence result for the non-linear model.

\begin{theorem}[Error estimate for the linear model {\cite[Theorem 2.28]{gdm}}]\label{th:gdm.error}
Let $\bar u$ be the solution to \eqref{ellgenf}, $\mathcal D$ be a GD, and $u$ be the solution to the gradient scheme \eqref{eq:gs.lin}.
Then, there exists $C$ depending only on $\Omega$ and $\underline{\lambda},\overline{\lambda}$ in \eqref{hyplambda} such that
\[
\norm{\bar u-\Pi_{\mathcal D}u}{L^2(\Omega)}+\norm{\nabla\bar u-\nabla_{\mathcal D}u}{L^2(\Omega)^d}
\le C(1+C_{\mathcal D})(W_{\mathcal D}(\Lambda\nabla\bar u+\bm{F})+S_{\mathcal D}(\bar u)).
\]
\end{theorem}

\begin{theorem}[Convergence for the nonlinear model {\cite[Theorem 2.9]{DE19}}]\label{th:gdm.cv.stefan}
Let $(\mathcal D_m)_{m\in\mathbb{N}}$ be a sequence of GDs which is consistent, limit-conforming and compact (which implies its coercivity \cite[Lemma 2.10]{gdm}), and such that each $\mathcal D_m$ has a piecewise constant reconstruction. 
Then, for any $m\in\mathbb{N}$ there exists a solution to \eqref{eq:stefan.gs} with $\mathcal D=\mathcal D_m$ and there exists a solution $\bar u$ to \eqref{eq:stefan.weak} such that, as $m\to\infty$, the following convergences hold:
\begin{align*}
&\Pi_{\mathcal D_m}u_m\to\bar u\mbox{ weakly in $L^2(\Omega)$,}\\
&\Pi_{\mathcal D_m}\zeta(u_m)\to\zeta(\bar u)\mbox{ strongly in $L^2(\Omega)$,}\\
&\nabla_{\mathcal D_m}\zeta(u_m)\to\nabla\zeta(\bar u)\mbox{ strongly in $L^2(\Omega)^d$}.
\end{align*}
\end{theorem}

The following lemma is particularly useful when considering mass-lumping of a given gradient discretisation. 
It shows that, under a simple assumption comparing the original and mass-lumped reconstructions, the properties of gradient discretisations that ensure the convergence of the gradient scheme are preserved.

\begin{lemma}[Mass-lumping preserves approximation properties {\cite[Th. 7.50]{gdm}}]\label{GDM:ml.is.ok}
Let $(\mathcal D_m)_{m\in\mathbb{N}}$ be a sequence of gradient discretisations that is coercive,
consistent, limit-conforming and compact. For each $m\in\mathbb{N}$ let $\mathcal D_m^*=(X_{\mathcal D_m,0}, \Pi_{\mathcal D_m}^*,\nabla_{\mathcal D_m})$ be a gradient discretisation that differs from $\mathcal D_m$ only through its function reconstruction. Assume the existence of a sequence $(\omega_m)_{m\in\mathbb{N}}$ of positive numbers such that $\omega_m\to 0$ as $m\to\infty$ and, for all $m\in\mathbb{N}$,
\[
\norm{\Pi_{\mathcal D_m}v-\Pi_{\mathcal D_m}^*v}{L^2(\Omega)}\le \omega_m\norm{\nabla_{\mathcal D_m}v}{L^2(\Omega)^d}\quad \forall v\in X_{\mathcal D_m,0}.
\]
Then, the sequence $(\mathcal D_m^*)_{m\in\mathbb{N}}$ is also coercive, consistent, limit-conforming and compact.
\end{lemma}

\subsection{Non-conforming gradient discretisations}

We recall here that polytopal non-conforming methods, as defined in Section \ref{sec:nc.approx.lin}, are gradient discretisation methods for gradient discretisations that satisfy the  properties required for the error estimates/convergence of the scheme. 

Let $V_{\mathfrak T,0}$ be a finite-dimensional subspace of $H^1_{\mathfrak T,0}$, and define the gradient discretisation $\mathcal D$ by:
\begin{equation}\label{eq:GD.for.NC}
	X_{\mathcal D,0}=V_{\mathfrak T,0}\,,\quad \Pi_{\mathcal D}v=v\mbox{ and }\nabla_{\mathcal D}v=\nabla_{\mathcal M}v\quad\forall v\in X_{\mathcal D,0}.
\end{equation}
Then, the non-conforming scheme \eqref{eq:nc.lin}, for the linear model, based on $V_{\mathfrak T,0}$ is the gradient scheme \eqref{eq:gs.lin} based on $\mathcal D$. Likewise, if $\Pi_{\mathfrak T}:V_{\mathfrak T,0}\to L^\infty(\Omega)$ is a piecewise-constant reconstruction of the form \eqref{eq:def.ml.nc} and $\mathcal D^*=(V_{\mathfrak T,0},\Pi_{\mathfrak T},\nabla_{\mathcal M})$, then the non-conforming scheme \eqref{eq:stefan.nc} for the Stefan/PME model is the gradient scheme \eqref{eq:stefan.gs} with $\mathcal D^*$ instead of $\mathcal D$. 
\begin{proposition}[Estimates for non-conforming methods {\cite[Proposition 9.5]{gdm}}]\label{prop:est.NC.GDM}
Let $\mathfrak T$ be a polytopal mesh and assume that $\gamma_{\mathfrak T}\le\gamma$. Let $V_{\mathfrak T,0}$ be a finite-dimensional subspace of $H^1_{\mathfrak T,0}$ and define the GD $\mathcal D$ by \eqref{eq:GD.for.NC}. Then, there exists $C>0$ depending only on $\Omega$ and $\gamma$ such that
\begin{align}\label{eq:nc.bound.CD}
C_{\mathcal D}\le{}& C\,\\
\label{eq:nc.bound.SD}
S_{\mathcal D}(\phi)\le{}& C\min_{v\in V_{\mathfrak T,0}}\norm{v-\phi}{H^1_{\mathfrak T,0}}\quad\forall \phi\in H^1_0(\Omega)\,,\\
\label{eq:nc.bound.WD}
W_{\mathcal D}(\bm{\psi})\le{}& C h_{\mathcal M}\norm{\bm{\psi}}{H^1(\Omega)^d}\quad\forall\bm{\psi}\in H^1(\Omega)^d.
\end{align}
\end{proposition}

\begin{theorem}[Properties of polytopal non-conforming methods {\cite[Th. 9.6]{gdm}}]\label{th:prop.generic.nc}
Let $(\mathfrak T_m)_{m\in\mathbb{N}}$ be a sequence of polytopal meshes such that $h_{\mathcal M_m}\to 0$ as $m\to\infty$ and $(\gamma_{\mathfrak T_m})_{m\in\mathbb{N}}$ is bounded. For each $m\in\mathbb{N}$ let $V_{\mathfrak T_m,0}$ be a finite-dimensional subspace of $H^1_{\mathfrak T_m,0}$ and assume that
\[
\min_{v\in V_{\mathfrak T_m,0}}\norm{v-\phi}{H^1_{\mathfrak T_m,0}}\to 0\mbox{ as $m\to\infty$,}\quad
\forall\phi\in H^1_0(\Omega).
\]
Then, the sequence $(\mathcal D_m)_{m\in\mathbb{N}}$ defined from $(V_{\mathfrak T_m,0})_{m\in\mathbb{N}}$ as in \eqref{eq:GD.for.NC} is coercive, consistent, limit-conforming, and compact.
\end{theorem}

\begin{remark}[Mass-lumped non-conforming method]
Combining this theorem with Lem\-ma \ref{GDM:ml.is.ok} shows that mass-lumped versions of polytopal non-conforming methods, such as the one presented in Section \ref{NCpoly:sec:ml}, usually also inherits the coercivity, consistency, limit-conformity and compactness properties. 
\end{remark}

\section{Perspectives}\label{sec:conclusion}

The LEPNC presented here is a low-order method. It is possible to extend this method into an arbitrary order approximation method. Let $k\ge 1$ be a sought approximation degree. For $K\in\mathcal M$, $\sigma\in\mathcal F_K$ and $q\in\mathbb{P}^{k-1}(\sigma)$, by the Riesz representation theorem in $L^2(\sigma)$ for the Lebesgue measure weighted by $\phi_{K,\sigma}$ (which is strictly positive on $\sigma$), there exists a unique $q_K\in\mathbb{P}^{k-1}(\sigma)$ such that
\begin{equation}\label{def:phiKsq}
\int_\sigma(\phi_{K,\sigma})_{|\sigma}q_Kr=\int_\sigma qr\,,\quad\forall r\in\mathbb{P}^{k-1}(\sigma).
\end{equation}
Set $\phi_{K,\sigma,q}=\phi_{K,\sigma}\hat{q}_K$, where $\hat{q}_K\in\mathbb{P}^{k-1}(K)$ is defined by $\hat{q}_K(\mathbi{x})=q_K(\pi_\sigma(\mathbi{x}))$ with $\pi_\sigma:\mathbb{R}^d\to H_\sigma$ the orthogonal projection on the hyperspace $H_\sigma$ spanned by $\sigma$. Then, the local $k$-degree LEPNC space is
\[
V_K^{\textsc{\tiny LEPNC}, k}:={\rm span}(\mathbb{P}^k(K)\cup\{\phi_{K,\sigma,q}\,:\,\sigma\in\mathcal F_K\,,\;q\in\mathbb{P}^{k-1}(\sigma)\}).
\]
For any set of moments of degree $\le k-1$ on $\sigma$, there exists $q\in\mathbb{P}^{k-1}(\sigma)$ that has the same moments and thus, in virtue of \eqref{def:phiKsq}, $\phi_{K,\sigma,q}$ also has these same moments on $\sigma$. Let $(\psi_{K,i})_{i=1,\ldots,n_k}$ be a basis of $\mathbb{P}^k(K)$. For each $i=1,\ldots,n_k$ we can find a linear combination $\sum_{\sigma\in\mathcal F_K}\phi_{K,\sigma,q_i}$ that has the same moments of degree $\le k-1$ as $\psi_{K,i}$ on each $\sigma\in\mathcal F_K$. The function $\psi_{K,i}-\sum_{\sigma\in\mathcal F_K}\phi_{K,\sigma,q_i}$ therefore has zero moments of degree $\le k-1$ on each face and, extended by 0 outside $K$, satisfies the $(k-1)$-degree patch test: its moments on each face coincide when viewed from each side of the faces.

When $\{K,L\}=\mathcal M_\sigma$, for a given $q\in\mathbb{P}^{k-1}(\sigma)$, by \eqref{def:phiKsq} the functions $\phi_{K,\sigma,q}$ and $\phi_{L,\sigma,q}$ have the same moments of degree $\le k-1$ on $\sigma$. Hence, in a similar way as in \eqref{NCpoly:def.phisigma}, we can glue $\phi_{K,\sigma,q}$ and $\phi_{L,\sigma,q}$ to obtain a global function that satisfies the $(k-1)$-degree patch test.

The family of these extended functions span a non-conforming space that has approximation properties of order $k$ (that is, \eqref{NCpoly:eq.approx.JD} holds with $\mathcal O(h_{\mathcal M}^{k+1})$ instead of $\mathcal O(h_{\mathcal M}^{2})$ in the right-hand side). The only caveat is the following: letting $(q_j)_{j=1,\ldots,\ell_k}$ be a basis of $\mathbb{P}^{k-1}(\sigma)$, the family $\{\psi_{K,i}\,:\,i=1,\ldots,n_K\}\cup \{\phi_{K,\sigma,q_j}\,:\,\sigma\in\mathcal F_K\,,\;j=1,\ldots,\ell_k\}$ spans the local space $V_K^{\textsc{\tiny LEPNC}, k}$; however, it is not clear if, in general, this family is linearly independent. Hence, describing a space of the local space (and, in consequence, the global space) requires to actually solve local linear problems, extracting a basis from a generating family.

\bibliographystyle{abbrv}
\bibliography{sema_simai_bib}

\begin{thebibliography}{10}

\bibitem{HArDCore2D}
{HArDCore2D -- Hybrid Arbitrary Degree::Core 2D}.
\newblock \url{https://github.com/jdroniou/HArDCore2D-release}, Version 2.0.2.

\bibitem{aav2002intro}
I.~Aavatsmark.
\newblock An introduction to multipoint flux approximations for quadrilateral
  grids.
\newblock {\em Comput. Geosci.}, 6(3-4):405--432, 2002.
\newblock Locally conservative numerical methods for flow in porous media.

\bibitem{Barenblatt52}
G.~I. Barenblatt.
\newblock On some unsteady motions of a liquid and gas in a porous medium.
\newblock {\em Akad. Nauk SSSR. Prikl. Mat. Meh.}, 16:67--78, 1952.

\bibitem{CDGGBP20}
C.~Canc\`es, J.~Droniou, C.~Guichard, G.~Manzini, M.~Bastisdas, and I.~S. Pop.
\newblock {\em Error estimates for the gradient discretisation method on
  degenerate parabolic equations of porous medium type}, pages 1--35.
\newblock SEMA-SIMAI, 2020.

\bibitem{chen1996equi}
Z.~Chen.
\newblock Equivalence between and multigrid algorithms for nonconforming and
  mixed methods for second-order elliptic problems.
\newblock {\em East-West J. Numer. Math.}, 4(1):1--33, 1996.

\bibitem{ciarlet}
P.~G. Ciarlet.
\newblock The finite element method for elliptic problems.
\newblock In {\em Studies in Mathematics and its Applications, Vol. 4}, pages
  xix+530. North-Holland Publishing Co., Amsterdam-New York-Oxford, 1978.

\bibitem{crouzeix-raviart-73}
M.~Crouzeix and P.-A. Raviart.
\newblock Conforming and nonconforming finite element methods for solving the
  stationary {S}tokes equations. {I}.
\newblock {\em Rev. Fran{\c c}aise Automat. Informat. Recherche
  Op{\'e}rationnelle S{\'e}r. Rouge}, 7(R-3):33--75, 1973.

\bibitem{hho-book}
D.~A. Di~Pietro and J.~Droniou.
\newblock {\em The Hybrid High-Order Method for Polytopal Meshes: Design,
  Analysis, and Applications}, volume~19 of {\em Modeling, Simulation and
  Applications}.
\newblock Springer International Publishing, 2020.

\bibitem{DPL15}
D.~A. Di~Pietro and S.~Lemaire.
\newblock An extension of the {C}rouzeix-{R}aviart space to general meshes with
  application to quasi-incompressible linear elasticity and {S}tokes flow.
\newblock {\em Math. Comp.}, 84(291):1--31, 2015.

\bibitem{DE19}
J.~Droniou and R.~Eymard.
\newblock High-order {M}ass-lumped {S}chemes for {N}onlinear {D}egenerate
  {E}lliptic {E}quations.
\newblock {\em SIAM J. Numer. Anal.}, 58(1):153--188, 2020.

\bibitem{gdm}
J.~Droniou, R.~Eymard, T.~Gallou\"et, C.~Guichard, and R.~Herbin.
\newblock {\em The gradient discretisation method}, volume~82 of {\em
  Mathematics \& Applications}.
\newblock Springer, 2018.

\bibitem{DEGH19}
J.~Droniou, R.~Eymard, T.~Gallou\"et, and R.~Herbin.
\newblock The gradient discretisation method for linear advection problems.
\newblock {\em Comput. Methods Appl. Math.}, page 23p, 2019.

\bibitem{DN16}
J.~Droniou and N.~Nataraj.
\newblock Improved {$L^2$} estimate for gradient schemes and super-convergence
  of the tpfa finite volume scheme.
\newblock {\em IMA J. Numer. Anal.}, page 40p, 2017.
\newblock To appear, DOI: 10.1093/imanum/drx028.

\bibitem{2Dbench}
R.~Herbin and F.~Hubert.
\newblock Benchmark on discretization schemes for anisotropic diffusion
  problems on general grids.
\newblock In {\em Finite volumes for complex applications {V}}, pages 659--692.
  ISTE, London, 2008.

\bibitem{mfdrev}
K.~Lipnikov, G.~Manzini, and M.~Shashkov.
\newblock Mimetic finite difference method.
\newblock {\em J. Comput. Phys.}, 257-Part B:1163--1227, 2014.

\bibitem{strang-fix}
G.~Strang and G.~Fix.
\newblock {\em An analysis of the finite element method}.
\newblock Wellesley-Cambridge Press, Wellesley, MA, second edition, 2008.

\bibitem{stum1979genpatch}
F.~Stummel.
\newblock The generalized patch test.
\newblock {\em SIAM Journal on Numerical Analysis}, 16(3):449--471, 1979.

\bibitem{Vazquez2007}
J.~V\'azquez.
\newblock {\em The porous medium equation: Mathematical Theory}.
\newblock Oxford Mathematical Monographs. The Clarendon Press Oxford University
  Press, 2007.

\bibitem{vohralik2007mixed}
M.~Vohral\'{\i}k, J.~Mary\v{s}ka, and O.~Sever\'{y}n.
\newblock Mixed and nonconforming finite element methods on a system of
  polygons.
\newblock {\em Appl. Numer. Math.}, 57(2):176--193, 2007.

\bibitem{wheeler2006multi}
M.~F. Wheeler and I.~Yotov.
\newblock A multipoint flux mixed finite element method.
\newblock {\em SIAM Journal on Numerical Analysis}, 44(5):2082--2106, 2006.

\bibitem{zien2014fin}
O.~C. Zienkiewicz, R.~L. Taylor, and D.~D. Fox.
\newblock {\em The finite element method for solid and structural mechanics}.
\newblock Elsevier/Butterworth Heinemann, Amsterdam, seventh edition, 2014.

\end{thebibliography}

\end{document}